\documentclass[a4paper]{amsart}

\usepackage{amscd,amsmath,amsthm,amssymb,amscd}
\usepackage{latexsym,amsbsy,mathrsfs}
\usepackage[enableskew,vcentermath]{youngtab}

\usepackage{txfonts}

\usepackage[bbgreekl]{mathbbol}

\usepackage{amsfonts,amssymb,amscd,verbatim, mathbbol}
\usepackage{bm,multicol}
\usepackage[pdftex,linkcolor=blue,citecolor=blue,backref=page]{hyperref}

\DeclareMathOperator{\End}{End}
\DeclareMathOperator{\Hom}{Hom}
\DeclareMathOperator{\Ext}{Ext}

\DeclareMathOperator{\soc}{soc}
\DeclareMathOperator{\lmod}{\!-mod}
\DeclareMathOperator{\gmod}{\!-gmod}

\DeclareMathOperator{\kor}{ker}

\def\Email#1{\email{{#1}}}

\def\rd{{\rm d}}
\def\lam{\lambda}
\def\Lam{\Lambda}

\def\co{\mathcal{O}}

\def\<{\langle}
\def\>{\rangle}
\def\Z{\mathbb{Z}}

\def\mcL{\mathscr{L}}
\def\mcR{\mathscr{R}}

\def\N{\mathbb N}
\def\C{\mathbb C}
\def\O{{\mathcal{O}}}
\def\fg{\mathfrak{g}}

\def\fn{\mathfrak{n}}
\def\fh{\mathfrak{h}}
\def\fb{\mathfrak{b}}
\def\lam{\lambda}
\def\Lam{\Lambda}

\def\mL{\mathcal{L}}
\def\HH{\mathbb{H}}
\def\ml{L}

\def\uH{\underline{H}}
\def\ucH{\underline{\mathcal{H}}}

\def\huH{\hat{\uH}}
\def\hucH{\hat{\ucH}}
\def\<{\langle}
\def\>{\rangle}

\DeclareMathOperator\ext{ext}

\DeclareMathOperator\hd{hd}

\DeclareMathOperator\id{id}

\DeclareMathOperator\adj{adj}

\theoremstyle{plain}
\numberwithin{equation}{section}
\newtheorem{cor}[equation]{Corollary}
\newtheorem{lem}[equation]{Lemma}
\newtheorem{thm}[equation]{Theorem}
\newtheorem{prop}[equation]{Proposition}
\theoremstyle{definition}
\newtheorem{dfn}[equation]{Definition}
\theoremstyle{remark}
\newtheorem{rem}[equation]{Remark}

\newtheorem*{eg*}{Example}

\title{On Hecke algebras and $\Z$-graded twisting, Shuffling and Zuckerman functors}

\author{Ming Fang}\address[Ming Fang]{Academy of Mathematics and Systems Science, Chinese Academy of Sciences,
100190 \\
\& School of Mathematical Sciences, University of Chinese Academy of Sciences,100049\\
 Beijing, P. R.~China}
\Email{fming@amss.ac.cn}

\author{Jun Hu}
\address[Jun Hu]{Key Laboratory of Algebraic Lie Theory and Analysis of Ministry of Education\\
School of Mathematics and Statistics\\ Beijing Institute of Technology\\ Beijing, 100081, P.R.~China}
\Email{junhu404@bit.edu.cn}

\author{Yujiao Sun}
\address[Yujiao Sun]{Key Laboratory of Mathematical Theory and Computation in Information Security, School of Mathematics and Statistics\\
	Beijing Institute of Technology\\
	Beijing, 100081, P.R. China}
\email{yujiao.sun@bit.edu.cn}

\thanks{The author is supported by the Natural Science Foundation of China (No.\ 12431002).}

\date{}

\begin{document}


\begin{abstract}
Let $\fg$ be a complex  semisimple Lie algebra with Weyl group $W$. Let $\HH(W)$ be the Iwahori-Hecke algebra associated to $W$. For each $w\in W$, let $T_w$ and $C_w$ be the corresponding $\Z$-graded twisting functor and $\Z$-graded shuffling functor respectively. In this paper we present a categorical action of $\HH(W)$ on the derived category $D^b(\O_0^\Z)$ of the $\Z$-graded BGG category $\O_0^\Z$ via derived twisting functors as well as a categorical action of $\HH(W)$ on $D^b(\O_0^\Z)$ via derived shuffling functors. As applications, we get graded character formulae for $T_sL(x)$ and $C_sL(x)$ for each simple reflection $s$. We describe the graded shifts occurring in the action of the $\Z$-graded twisting and shuffling functors on dual Verma modules and simple modules. We also characterize the action of the derived $\Z$-graded Zuckerman functors on simple modules.
\end{abstract}



\maketitle
\section{Introduction}

Let $\fg$ be a finite dimensional complex semisimple Lie algebra with a fixed triangular decomposition $\fg=\fn\oplus\fh\oplus\fn^{-}$. Let $\O$ be the associated BGG category as defined in \cite{Hum}. For each $\lam\in\fh^*$, we use $L(\lam)$, $\Delta(\lam)$, $\nabla(\lam)$ and $P(\lam)$ to denote the simple module, the Verma module, the dual Verma module and the indecomposable projective module in $\O$ labelled by $\lam$ respectively.

Let $\Phi$ be the root system of $\fg$ and $W$ the Weyl group of $\fg$. Let $S$ be the set of simple reflections in $W$. For each $\lam\in\fh^*$ and $x\in W$, we define $x\cdot\lam:=x(\lam+\rho)-\rho$, where $\rho$ denotes the half sum of all the positive roots in $\Phi$.  We use $\O_\lam$ to denote the Serre subcategory of $\O$ generated by all $L(w\cdot\lam)$ for $w\in W$. In this paper we are mainly concerned with the regular block $\O_0$. By construction, $\oplus_{x\in W}P(x\cdot 0)$ is a regenerator of $\O_0$. We define $$
A:=\Biggl(\End_{\O_0}\biggl(\bigoplus_{x\in W}P(x\cdot 0)\biggr)\Biggr)^{{\rm op}} .
$$

By \cite{So1}, $A$ is a finite dimensional quasi-hereditary (basic) $\C$-algebra in the sense of \cite{CPS}, and there is an equivalence of categories: $\O_0\cong A\lmod$, where $A\lmod$ denotes the category of finite dimensional left $A$-modules. Moreover, by \cite{BGS}, we know that $A$ can be endowed with a Koszul $\Z$-grading which makes it into a Koszul algebra. Thus the category $A\gmod$ of finite dimensional $\Z$-graded left $A$-modules can be regarded as a $\Z$-graded version $\O_0^\Z$ of the BGG category $\O_0$. Henceforth, we set $$
\O_0^\Z:=A\gmod .
$$
For any $\Z$-graded module $M$ and $k\in\Z$, we define a $\Z$-graded module $M\<k\>$ such that $(M\<k\>)_i:=M_{i-k}$, $\forall\,i\in\Z$.\footnote{Note that we use an opposite convention for the grading shift as in \cite{KMM}.} All the structural modules (such as simple module $L(x\cdot 0)$, Verma module $\Delta(x\cdot 0)$ and indecomposable projective module $P(x\cdot 0)$) admit graded lifts. We fix a unique $\Z$-graded lift $L(x)$ of the simple module $L(x\cdot 0)$ such that $L(x)$ is concentrated in degree $0$; we fix a unique $\Z$-graded lift $\Delta(x)$ of the Verma module $\Delta(x\cdot 0)$ such that the unique simple head of $\Delta(x)$ is isomorphic to $L(x)$; we fix a unique $\Z$-graded lift $P(x)$ of the indecomposable projective module $P(x\cdot 0)$ such that the unique simple head of $P(x)$ is isomorphic to $L(x)$. Let ``$\circledast$'' be the $\Z$-graded duality functor on $\O_0^\Z$ introduced in \cite{Hu}. We define $\nabla(x):=\Delta(x)^{\circledast}$, which gives a $\Z$-graded lift of the dual Verma module $\nabla(x\cdot 0)$.

Twisting functors were first introduced in \cite{Ar}. These functors allow $\Z$-graded lifts, see \cite[Appendix]{MO}. For each $x\in W$, we use $T_x$ to denote the corresponding $\Z$-graded twisting functor. Shuffling functors were first introduced in \cite{Ca} and studied in \cite{Ir93} and \cite{MaStro05}. By \cite[\S2.7]{CMZ}, these functors allow $\Z$-graded lifts. For each $x\in W$, we use $C_x$ to denote the corresponding $\Z$-graded shuffling functor.

Let $v$ be an indeterminate over $\Z$ and $q:=v^2$. We use ``$\leq$'' to denote the Bruhat partial order on $W$. That is, for any $x,y\in W$, $x\leq y$ if and only if $x=s_{i_{j_1}}\cdots s_{i_{j_t}}$ for some reduced expression $y=s_{i_1}\cdots s_{i_m}$ of $y$ and some integers $1\leq t\leq m$, $1\leq j_1<\cdots< j_t\leq m$, where $s_{i_j}\in S$ for each $j$. If $x\leq y$ and $x\neq y$ then we write $x<y$. Let $w_0$ be the unique longest element in $W$.

\begin{dfn} The Iwahori-Hecke algebra $\HH(W)=\HH(W,S)$ with Hecke parameter $v$ associated to $(W,S)$ is a free $\Z[v,v^{-1}]$-module with standard basis $\{H_w|w\in W\}$ and multiplication rule given by $$
H_xH_y=H_{xy},\,\,\text{if $\ell(xy)=\ell(x)+\ell(y)$,}\quad H_s^2=(v^{-1}-v)H_s+H_e,\,\,\forall\,s\in S ,
$$
where $H_e$ is the identity element of $\HH(W)$.
\end{dfn}
The Hecke algebra $\HH(W)$ is a $v$-deformation of the group ring $\Z[W]$. One should identify $v$ in this paper with $v^{-1}$ (resp., $u^{-1/2}$) in the notation of \cite{KL} (resp., of \cite{L0}), and $H_w$ in this paper with the element $v^{-\ell(w)}T_w$ (resp., $u^{-\ell(w)/2}T_w$) in the notation of \cite{KL} (resp., of \cite{L0}). The following theorem is the first main result of this paper.

\begin{thm}\label{mainthm1} Let $\rho$ be the $\Z[v,v^{-1}]$-module isomorphism from the Grothendieck group of $D^b(\O_0^\Z)$ onto $\HH(W)$
defined by $$
\rho\Bigl([\nabla(x)\<k\>]\Bigr):=v^k H_{w_0x^{-1}},\quad \forall\,x\in W, k\in\Z .
$$
Then the derived twisting functors $\mL T_x$ gives rise to a categorical action of the Iwahori-Hecke algebra $\HH(W)$ on $D^b(\O_0^\Z)$ such that $$
\rho\bigl([(\mL T_x)M]\bigr)=\rho([M])H_x,\quad\forall\,x\in W, M\in\O_0^\Z .
$$
In particular, $\rho\bigl([(\mL T_x)\nabla(y)]\bigr)=H_{w_0y^{-1}}H_x,\forall\,x,y\in W$. Moreover, $\rho\bigl([L(x)]\bigr)=\ucH_{w_0x^{-1}}$. If furthermore $x\in W$ is an involution then $$
\rho\bigl([\Delta(x)]\bigr)=H_{xw_0}^{-1},\,\, $$
where $\ucH_{w_0x^{-1}}$ is the twisted Kazhdan-Lusztig basis element corresponding to $w_0x^{-1}$ (see Section 2).
\end{thm}

Let $s\in S$ and $x\in W$. It is well-known that $T_sL(x)\neq 0$ if and only if $sx<x$. Andersen and Stroppel \cite{AS} studied the structure of $T_sL(x)$ in the ungraded setting. Using Theorem \ref{mainthm1}, we obtained two
graded character formulae for the twisting simple module $T_sL(x)$ in terms of Kazhdan-Lusztig polynomials, which is the second main result of this paper.

\begin{thm}\label{mainthm2} Let $s\in S, x\in W$ with $sx<x$. Then we have $\hd(T_sL(x))\cong L(x)\<-1\>$ and $[\soc T_sL(x):L(sx)]_v=1$. Moreover, in the Grothendieck group of $\O_0^\Z$, $$\begin{aligned}
{} [T_sL(x)] {} &=v^{-1}[L(x)]+[L(sx)]+\sum_{\substack{y\in W, sy>y>x\\ \mu(x,y)\neq 0}}\mu(x,y)[L(y)]\\
&=\sum_{\substack{y\geq x\\ x\nleq sy<y}}(-v)^{\ell(x)-\ell(y)}P_{w_0y^{-1},w_0x^{-1}}(v^2)[\nabla(sy)]+\sum_{\substack{y\geq x\\ sy>y}}(-v)^{\ell(x)-\ell(y)+1}P_{yw_0,xw_0}(v^2)\bigl([\nabla(y)]-v^{-1}[\nabla(sy)]\bigr) .
\end{aligned}
$$
where $\mu(x,y)$ is the ``leading coefficient'' for Kazhdan-Lusztig polynomial $P_{x,y}(q)$ (see Section 2 for precise definition).
\end{thm}

For each $s\in S$, let $Z_s$ be the $\Z$-graded Zuckerman functor associated to $s$ (see \cite[\S6.1]{Ma2}, \cite[\S3]{HX23}). Recall that $\mL_j Z_s=0$ for any $j>2$.
Set $\hat{Z}_s:=\circledast\circ Z_s\circ\circledast$, the $\Z$-graded dual Zuckerman functor. Then $\mathcal{R}_j\hat{Z}_s=0$ for any $j>2$. Our third main result of this paper below gives an algorithm to compute the graded character of $T_sM$ for any $M\in\O_0^\Z$.

\begin{thm}\label{mainthm3} Let $s\in S$.

(1) For any $x\in W$, we have $\mL_2 Z_sL(x)=\begin{cases} L(x)\<1\>, &\text{if $sx>x$;}\\ 0, &\text{if $sx<x$.}\end{cases}$. If $sx>x$ then $\mathcal{L}_1Z_sL(x)=0$; if $sx<x$, then $$
\begin{aligned}
{} [\mathcal{L}_1Z_sL(x)]&=v[\Delta(sx)]-v^2[\Delta(x)]-\sum_{\substack{z\in W\\ sz<z>x}}v^{\ell(z)-\ell(x)}P_{x,z}(v^{-2})[\mL_1Z_s L(z)]+(v+1)\sum_{\substack{z\in W\\ sz>z>x}}v^{\ell(z)-\ell(x)}P_{x,z}(v^{-2})[L(z)].
\end{aligned}
$$

(2) Let $M\in\O_0^\Z$. Suppose that in the Grothendieck group of $\O_0^\Z$, $$
[M/\hat{Z}_s(M)]=\sum_{x\in W}c_x(v,v^{-1})[L(x)], $$ where $c_x(v,v^{-1})\in\mathbb{N}[v,v^{-1}]$ for each $x\in W$.
Then in the Grothendieck group of $\O_0^\Z$ we have $$
[T_sM]=\sum_{\substack{x\in W\\ sx<x}}c_x(v,v^{-1})[T_sL(x)]-\sum_{\substack{x\in W\\ sx>x}}vc_x(v,v^{-1})[L(x)] .
$$
\end{thm}

Our fourth main result of this paper gives an analogue of Theorem \ref{mainthm1} for $\Z$-graded shuffling functors.

\begin{thm}\label{mainthm4} Let $\rho$ be the $\Z[v,v^{-1}]$-module isomorphism from the Grothendieck group of $D^b(\O_0^\Z)$ onto $\HH(W)$
defined by $$
\rho\Bigl([\nabla(x)\<k\>]\Bigr):=v^k H_{w_0x},\quad \forall\,x\in W, k\in\Z .
$$
Then the derived shuffling functors $\mL C_x$ gives rise to a categorical action of the Iwahori-Hecke algebra $\HH(W)$ on $D^b(\O_0^\Z)$ such that $$
\rho\bigl([(\mL C_x)M]\bigr)=\rho([M])H_x,\quad\forall\,x\in W, M\in\O_0^\Z .
$$
In particular, $\rho\bigl([(\mL C_x)\Delta(y)]\bigr)=H_{w_0y}H_x,\forall\,x,y\in W$. Moreover, $\rho\bigl([L(x)]\bigr)=\ucH_{w_0x}$.
\end{thm}

Our fifth main result of this paper presents a graded character formula for the shuffling simple module $C_sL(x)$ in terms of Kazhdan-Lusztig polynomials.

\begin{prop}\label{mainprop6} Let $s\in S$ and $x\in W$. Then $C_sL(x)\neq 0$ if and only if $xs<x$. If $xs<s$, then $\mathcal{L}_1C_sL(x)=0$, and in the Grothendieck group of $\O_0^\Z$, $$\begin{aligned}
{} [C_sL(x)] {} &=v^{-1}[L(x)]+[L(xs)]+\sum_{\substack{y\in W, ys>y>x\\ \mu(x,y)\neq 0}}\mu(x,y)[L(y)]\\
&=\sum_{\substack{y\geq x\\ x\nleq ys<y}}(-v)^{\ell(x)-\ell(y)}P_{w_0y,w_0x}(v^2)[\nabla(ys)]+\sum_{\substack{y\geq x\\ ys>y}}(-v)^{\ell(x)-\ell(y)+1}P_{w_0y,w_0x}(v^2)\bigl([\nabla(y)]-v^{-1}[\nabla(ys)]\bigr) .
\end{aligned}
$$
\end{prop}


The content is organised as follows. In Section 2 we first recall some preliminary results on the BGG category $\O$ as well as its $\Z$-graded analogue, and some basic property on the twisting functors and their $\Z$-graded lift.  Then we recall the Kazhdan-Lusztig basis, twisted Kazhdan-Lusztig basis and their dual bases following \cite{KL} and \cite{L0}. We also recall the categorification of Hecke algebras using indecomposable projective functors in Lemma \ref{keylem0} as well as its Ringel dual version in Lemma \ref{keylem02}. In Section 3 we explicitly describe the graded shifts occurring in the action of the $\Z$-graded twisting functors on dual Verma modules and simple modules in Lemmas \ref{Tsaction1}, \ref{Tsaction2}. Then we give the proof of our first main result Theorem \ref{mainthm1}. Using Theorem \ref{mainthm1}, we then give the proof of the second main result Theorem \ref{mainthm2} in the same section, which gives two $\Z$-graded character formulae of $T_sL(x)$ for each simple reflection $s$. We explicitly describe the action of the second derived $\Z$-graded Zuckerman functors on simple modules in Lemma \ref{keylem3}, and presents a recursive formula to calculate the action of the first derived $\Z$-graded Zuckerman functors on simple modules in the Grothendieck group in Lemma \ref{0lem}. The third main result Theorem \ref{mainthm3} gives an algorithm to compute $T_sM$ in the Grothendieck group for any $M\in\O_0^\Z$. In Section 4, we first describe in Lemmas \ref{Csaction0} the action of $\Z$-graded shuffling functors on Verma modules and dual Verma modules. Then we give the proof of our fourth and fifth main results Theorems \ref{mainthm4}, \ref{mainprop6} which generalize Theorem \ref{mainthm1}, Theorem \ref{mainthm2} to the $\Z$-graded Shuffling functors case.

\section{Preliminary}

Let $\fg$ be a finite dimensional complex semisimple Lie algebra with a triangular decomposition $\fg=\fn\oplus\fh\oplus\fn^{-}$, where $\fh$ is a fixed Cartan subalgebra and $\fb:=\fh\oplus\fn$ is the corresponding Borel subalgebra. Let $U(\fg), U(\fn)$ be the universal enveloping algebra of $\fg$ and $\fn$ respectively. The BGG category $\O$ is the full subcategory of the category of $U(\fg)$-module which consists of all finitely generated $U(\fg)$-module $M$ satisfying the following conditions:\begin{enumerate}
\item[1)] $M$ has a weight space decomposition $M=\oplus_{\lam\in\fh^*}M_\lam$, where $M_\lam:=\{v\in M|hv=\lam(h)v,\forall\,h\in\fh\}$; and
\item[2)] the action of $U(\fn)$ on $M$ is locally finite .
\end{enumerate}


Let $\Pi^\vee:=\{\alpha^\vee|\alpha\in\Phi\}$ be the set of simple coroots. Let $\Lam$ be the set of integral weights. That is, $$
\Lam:=\bigl\{\lam\in\fh^*\bigm|\<\alpha^\vee,\lam\>\in\Z,\forall\,\alpha\in\Pi\bigr\} .
$$
We consider the integral part $\O_\Lam$ of the BGG category $\O$ which consists of all modules in $\O$ with weights in $\Lam$. We use $\Lam/(W,\cdot)$ to denote the set of orbits on $\Lam$ under the dot action of $W$.
There is a block decomposition as follows: $$
\O_\Lam=\bigoplus_{\lam\in\Lam/(W,\cdot)}\O_\lam .
$$
In this paper, we shall only be interested in the regular integral block $\O_0$. For any finite dimensional $\fg$-module $V$, we shall call any direct summand of a functor of the form $-\otimes V$ a projective functor. By \cite[Theorem 3.3]{BG}, isomorphism classes of indecomposable projective endofunctor on $\O_0$ are in bijection with elements in $W$. More precisely, for each $w\in W$, there is a unique (up to isomorphism) indecomposable projective endofunctor $\theta_w: \O_0\rightarrow\O_0$ such that $\theta_w(\Delta(0))\cong P(w)$. For each $w\in W$, the functors $\theta_w, \theta_{w^{-1}}$ are biadjoint to each other. Moreover, by \cite{BG}, the projective functor $\theta_w$ preserve both $\mathcal{F}(\Delta)$ and $\mathcal{F}(\nabla)$ and hence the category of tilting modules, where $\mathcal{F}(\Delta)$ (resp., $\mathcal{F}(\nabla)$) denotes the full subcategory of $\O$ which consists of all modules having a $\Delta$-flag (resp., having a $\nabla$-flag).

\begin{dfn} For each $w\in W$, we use $e_w$ to denote the unique degree $0$ homogeneous primitive idempotent in $A$ corresponding to $L(w)$. That is, $e_w$ corresponds to the projection from $\bigoplus_{x\in W}P(x)$ onto $P(w)$.
\end{dfn}

For each $w\in W$, the projective functor $\theta_w$ also admits a graded lift which will be denoted by the same notation $\theta_w$. We use $\mathcal{F}(\mathbb{\Delta})$ to denote the full subcategory of $\O_0^\Z$ which consists of all modules having a $\Z$-graded $\Delta$-flag (i.e., a filtration in $A\gmod$ such that each successive quotient being isomorphic to some modules of the form ${\Delta}(w)\<k\>$) for some $w\in W$ and $k\in\Z$). Then by \cite{Stro2}, the $\Z$-graded projective functor $\theta_w$ preserves the subcategory $\mathcal{F}(\mathbb{\Delta})$.

Let $S\subset W$ be the set of simple reflections in $W$. The set $S$ generates the Weyl group $W$. A word $w=s_{i_{1}}s_{i_{2}}\ldots s_{i_{k}}$, where $s_{i_a}\in S$ for each $1\leq a\leq k$, is called a reduced expression of $w$ if $k$ is minimal; in this case we say that $w$ has length $k$ and we write $\ell(w)=k$. For each $s\in S$, let $T_s$ be the corresponding twisting functor, see e.g., \cite{Ar, AS, KhM}. Recall that twisting functors are right exact and they satisfy braid relations (\cite[Theorem 2]{KhM}), which allows us to define (up to isomorphism of functors) \begin{equation}\label{braid1}
T_w:=T_{s_{i_1}}T_{s_{i_2}}\cdots T_{s_{i_k}},
\end{equation}
where $s_{i_1}s_{i_2}\cdots s_{i_k}$ is a reduced expression of $w$. By \cite[Lemma 2.1(5)]{AS}, each twisting functor $T_s$ commutes with any projective functor $\theta$. It follows that for each $w\in W$, the functor $T_w$ commutes with the projective functor $\theta$ as well. That is, $T_w\circ\theta\cong\theta\circ T_w$. For each $w\in W$, the twisting functor $T_w$ is right exact. For each $i\in\N$, we use $\mL_iT_w$ to denote the $i$th left derived functor of $T_w$.

\begin{lem}\text{(\cite[Theorem 2.2]{AS})}\label{lemacylic1} For any $s\in S$ and $i>1$, we have $\mL_iT_s=0$. Moreover, for any $w\in W$, $x\in W$ and $j>0$, we have $\mL_jT_w\Delta(w)=0$.
\end{lem}

\begin{cor}\label{mltwisted} For any $x,y\in W$ and $j\in\N$, we have $\theta_x\circ(\mL_j T_y)\cong(\mL_j T_y)\circ\theta_x$.
\end{cor}

\begin{proof} Since twisting functor commutes with the projective functor, it follows that the corollary holds for $j=0$. Since $\theta_x$ is an exact and sends projective to projective, we can thus deduce that $\theta_x\circ(\mL T_y)\cong(\mL T_y)\circ\theta_x$ as functors on $D^b(\O_0^\Z)$, from which we see the corollary holds for all $j\in\N$.
\end{proof}

\begin{cor}\label{acuclicor} Let $w_1,w_2\in W$ with $\ell(w_1w_2)=\ell(w_1)+\ell(w_2)$. Then for any exact complex in $P^\bullet\in K^+({\rm Proj})$, $T_{w_2}P^\bullet$ is again an exact complex and an acylic complex for the functor $T_{w_1}$.
\end{cor}

\begin{proof} Since each graded projective module $P$ has a $\Z$-graded $\Delta$-filtration, it follows from Lemma \ref{lemacylic1} that $(\mL_1T_w)M=0$ and $T_w$ is right exact that $T_{w_1}P^\bullet$ is again an exact complex. As a result, for any exact complex $P^\bullet\in K^+({\rm Proj})$, we see that $\mL_j(T_{w_1}T_{w_2}P^\bullet)=\mL_j(T_{w_1w_2}P^\bullet)=0$ for any $j>0$. Note that $\mL_0(T_{w_1w_2}P^\bullet)=0$ holds because $P^\bullet$ is exact and $T_{w_1w_2}$ is right exact. This proves the corollary.
\end{proof}

For each $s\in S$, the twisting functor $T_s$ has a right adjoint $G_s$---the Joseph's completion functor. For a reduced expression $s_{i_1}s_{i_2}\cdots s_{i_k}$ of $w\in W$, we define $G_w:=G_{s_{i_k}}\cdots G_{s_{i_2}}G_{s_{i_1}}$. Then $G_w$ is a right adjoint of $T_w$. By \cite[Theorem 4.1]{AS}, we have $G_w\cong d\circ T_{w^{-1}}\circ d$, where $d:A\lmod\rightarrow A\lmod$ is the (ungraded) duality functor induced from the duality functor ``$\vee$'' on $\O_0$ (see \cite[\S3.2]{Hum}).

By \cite[Appendix]{MO}, each twisting functor $T_w$ allows a $\Z$-graded lift. In this paper, we shall follow the following formulation given in \cite[Appendix]{MO} to define the $\Z$-graded lift of the twisting functor $T_x$.
Henceforth, we shall use the same letter $T_x$ to denote the above-defined $\Z$-graded lift of the twisting functor $T_x$. The functors
$\circledast\circ T_{w^{-1}}\circ\circledast$ gives a $\Z$-graded lift of the functor $G_w$ and it is a right adjoint of $T_w$. By abuse of notation, we shall denote it by $G_w$ again throughout this paper. As a result, we have a $\Z$-graded space isomorphism: \begin{equation}\label{biadj}
\Hom_A\bigl(T_wM, N\bigr)\cong\Hom_A\bigl(M,G_wN\bigr),\quad \forall\, M,N\in A\lmod .
\end{equation}

Let $D^b(\O_0^\Z)$ be the bounded derived category of finite dimensional $\Z$-graded $A$-modules. For each $j\in\Z$, let ``$[j]$'' be the functor of shifting the position in a complex defined as follows: $X[j]^i:=X^{i+j}$, $\forall\,i\in\Z, X^\bullet\in D^b(\O_0^\Z)$. Each projective functor $\theta_w$ can also be regarded as a functor on $D^b(\O_0^\Z)$. Let $\rd: D^b(\O_0^\Z)\rightarrow D^b(\O_0^\Z)$ be the duality functor which is induced from the $\Z$-graded duality functor $\circledast: \O_0^\Z\rightarrow\O_0^\Z$. Recall by \cite{BGS}, $\O_0$ is Ringel self-dual. The derived twisting functor $\mL T_{w_0}$ gives the Ringel duality auto-equivalence of $D^b(\O_0^\Z)$ such that \begin{equation}\label{RingelDuality}
P(w)\mapsto T(w_0w),\quad T(w)\mapsto I(w_0w),\quad \Delta(w)\mapsto \nabla(w_0w),\quad\forall\, w\in W.
\end{equation}

Let ``$*$'' be the unique $\Z[v,v^{-1}]$-linear anti-involution of $\HH(W)$ which is uniquely determined by $H_w^*:=H_{w^{-1}}$ for any $w\in W$.

There is a unique $\Z$-linear involution ``$-$'' (called bar involution) on $\HH(W)$ which maps $v^k$ to $v^{-k}$ for all $k\in\Z$ and $H_w$ to $H_{w^{-1}}^{-1}$ for all $w\in W$. By a well-known result of Kazhdan and Lusztig \cite{KL}, $\HH(W)$ has a unique $\Z[v,v^{-1}]$-basis $\{\uH_w|w\in W\}$, and a unique $\Z[v,v^{-1}]$-basis $\{\ucH_w|w\in W\}$ such that \begin{enumerate}
\item[1)] for each $w\in W$, $\overline{\uH_w}=\uH_w$, $\overline{\ucH_w}=\ucH_w$, and
\item[2)] we have \begin{equation}\label{uch}
\uH_w=H_w+\sum_{w>y\in W}v^{\ell(w)-\ell(y)}P_{y,w}(v^{-2})H_y ,\quad\,
\ucH_w=H_w+\sum_{w>y\in W}(-v)^{\ell(y)-\ell(w)}P_{y,w}(v^{2})H_y ,
\end{equation}
where $P_{y,w}(q)$ is a polynomial in $q$ of degree $\leq (\ell(w)-\ell(y)-1)/2$, and $P_{w,w}(q):=1$.
\end{enumerate}
In particular, \begin{equation}\label{uh11}
\uH_w\in H_w+\sum_{w>y\in W}v\Z[v]H_y,\quad \ucH_w\in H_w+\sum_{w>y\in W}v^{-1}\Z[v^{-1}]H_y. \end{equation}
The polynomial $P_{y,w}(v^2)$ can be identified with $P_{y,w}(u^{-1})$ in the notation of \cite[Chapter 5]{L0}, the basis elements $\uH_w, \ucH_w$ can be identified with $C'_w, C_w$ in the notation of \cite[Chapter 5]{L0} with $u$ there replaced with $v^{-2}$.
We call $\{\uH_w|w\in W\}$ the Kazhdan-Lusztig basis of $\HH(W)$, and $\{\ucH_w|w\in W\}$ the twisted Kazhdan-Lusztig basis of $\HH(W)$. In particular, $\uH_s=H_s+v$, $\ucH_s=H_s-v^{-1}$ for each $s\in S$.

Let $x,y\in W$ with $x\leq y$. By the last paragraph we see that $\deg P_{x,y}(q)\leq (\ell(y)-\ell(x)-1)/2$. Let $\mu(x,y)$ be the coefficient of $q^{(\ell(y)-\ell(x)-1)/2}$ in $ P_{x,y}(q)$. We call $\mu(x,y)$ the ``leading coefficient'' of $P_{x,y}(q)$. If $y\leq x$, then we define $\mu(x,y):=\mu(y,x)$. By \cite[(2.3.b),(2.3.c)]{KL} and \cite[Theorem 6.6]{L1}, we have \begin{equation}\label{2Hs}
\uH_w \uH_s =\begin{cases} \uH_{ws}+\sum_{\substack{y\in W\\ ys<y<w}}\mu(y,w)\uH_{y}, &\text{if $ws>w$;}\\
(v+v^{-1})\uH_w, &\text{if $ws<w$.}\end{cases},\,\,
\ucH_w \ucH_s =\begin{cases} \ucH_{ws}+\sum_{\substack{y\in W\\ ys<y<w}}\mu(y,w)\ucH_{y}, &\text{if $ws>w$;}\\
0, &\text{if $ws<w$.}\\
\end{cases}
\end{equation}

\begin{lem}\text{(\cite{KL}, \cite{Bre}, \cite{Wa})}\label{Pmu} Let $x,y\in W$ with $x\leq y$. The we have $$
P_{x,y}(q)=P_{x^{-1},y^{-1}}(q)=P_{yw_0,xw_0}(q)=P_{w_0y,w_0x}(q),\quad \mu(x,y)=\mu(x^{-1},y^{-1})=\mu(yw_0,xw_0)=\mu(w_0y,w_0x),
$$
\end{lem}

Following \cite[Chapter 5]{L0}, we set $Q_{w,y}:=P_{w_0y,w_0w}$, $\forall\,w\leq y$. For any $x\in W$, we set \begin{equation}\label{uch2}
\huH_w=H_w+\sum_{w<y\in W}(-v)^{\ell(y)-\ell(w)}Q_{w,y}(v^{-2})H_y ,\quad\,
\hucH_w=H_w+\sum_{w<y\in W}v^{\ell(w)-\ell(y)}Q_{w,y}(v^{2})H_y ,
\end{equation}
In particular, \begin{equation}\label{uch22}
\huH_w\in H_w+\sum_{w<y\in W}v\Z[v]H_y,\quad \hucH_w\in H_w+\sum_{w<y\in W}v^{-1}\Z[v^{-1}]H_y. \end{equation}
The elements $\huH_w, \hucH_w$ can be identified with $D'_w, D_w$ in the notation of \cite[Chapter 5]{L0} with $u$ there replaced with $v^{-2}$.

Let $\tau: \HH(W)\rightarrow\Z[v,v^{-1}]$ be the linear function on $\HH(W)$ defined by $\tau(\sum_{w\in W}r_wH_w)=r_e$, where $r_w\in\Z[v,v^{-1}]$ for each $w\in W$. It is well-known that $\tau$ is a non-degenerate symmetrizing form on $\HH(W)$. By \cite[(5.1.10)]{L0}, we have \begin{equation}
\tau(\uH_x\hat{\uH}_{y^{-1}})=\delta_{xy}=\tau(\ucH_x\hat{\ucH}_{y^{-1}}),\quad \forall\,x,y\in W .
\end{equation}
For this reason, we call $\{\hat{\uH}_w|w\in W\}$ the dual Kazhdan-Lusztig basis of $\HH(W)$, and $\{\hat{\ucH}_w|w\in W\}$ the dual twisted Kazhdan-Lusztig basis of $\HH(W)$. Applying Lemma \ref{Pmu} and \cite[(5.1.8)]{L0}, we can get that \begin{equation}\label{4KLbasis1}
\huH_w=\ucH_{ww_0}H_{w_0}=H_{w_0}\ucH_{w_0w},\quad \hucH_w=\uH_{ww_0}H_{w_0}=H_{w_0}\uH_{w_0w}.
\end{equation}

We use $\Ext^1(-,-)$ (resp., $\ext_A^1(-,-)$) to denote the extension functor in the category $\O_0$ (resp., the category $\O_0^\Z$).

\begin{lem}\label{parity1} Let $x,y\in W$. Then $\dim\Ext^1_\O(L(x),L(y))=\mu(x,y)=\dim\ext_A^1(\ml(x),\ml(y)\<1\>)$. In particular, $\ext_A^1(\ml(x),\ml(y)\<k\>)\neq 0$ only if $k=1$, $\ell(x)\equiv\ell(y)+1\pmod{2\Z}$ and either $x<y$ or $y<x$.
\end{lem}

\begin{proof} This follows from  \cite[Theorem 8.15]{Hum} and \cite[Proposition 2.1.3]{BGS} and \cite[Fact 3.1]{Wa}.
\end{proof}

For any $w\in W$, we define $$
\mcL(w):=\{s\in S|sw<w\},\quad \mcR(w):=\{s\in S|ws<w\}.
$$

\begin{lem}\label{3prop} Let $x,y\in W$. \begin{enumerate}
\item[(i)] If $\ell(y)-\ell(x)=1$, then $P_{x,y}(q)=1=\mu(x,y)$;
\item[(ii)] If $x\leq y$ and either $\mcL(y)\not\subseteq\mcL(x)$ or $\mcR(y)\not\subseteq\mcR(x)$, then $\mu(x,y)\neq 0$ only if $\ell(y)-\ell(x)=1$.
\end{enumerate}
\end{lem}

\begin{proof} (i) follows from \cite[Lemma 2.6(iii)]{KL}. (ii) follows from \cite[Fact 3.2]{Wa}.
\end{proof}

For any category $\mathcal{C}$, we use $[\mathcal{C}]$ to denote the Grothendieck group of $\mathcal{C}$. For any $k\in\Z$ and $M\in\mathcal{O}_0^\Z$, we define $$
v^k[M]:=[M\<k\>] .
$$
Then $[\O_0^\Z]$ naturally becomes a $\Z[v,v^{-1}]$-module.

\begin{lem}[\text{\cite[Propositions 7.10, 7.11]{Ma2}}]\label{keylem0} There is a unique $\Z[v,v^{-1}]$-module isomorphism: $\varphi: \HH(W)\cong\bigl[\O_0^\Z\bigr]$ such that $$
\varphi(H_w)=[\Delta(w)],\quad \varphi(\uH_w)=[P(w)],\quad \varphi(\huH_w)=[L(w)],   \,\forall\,w\in W,
$$
and the following diagram commutes: $$
\begin{CD}
    \HH(W) @>\cdot \uH_w>> \HH(W)\\
    @V \varphi V V @VV \varphi V\\
    [\O_0^\Z] @>>[\theta_w]\cdot> [\O_0^\Z]
\end{CD} $$\end{lem}

The following result seems to be well-known to experts but not explicitly stated anywhere in the literature. We add it for completeness.

\begin{lem} For any $w\in W$, we have $$
\varphi(\hucH_w)=[T(w)] .
$$
\end{lem}

\begin{proof} By graded Ringel self-duality (\ref{RingelDuality}), we have $$
\bigl(T(w):\nabla(y)\bigr)_v=\bigl(P(w_0w):\Delta(w_0y)\bigr)_v=v^{\ell(y)-\ell(w)}P_{w_0y,w_0w}(v^{-2}) .
$$
Since $T(w)^{\circledast}\cong T(w)$ and $\nabla(y)^{\circledast}\cong\Delta(y)$, it follows that $$
\bigl(T(w):\Delta(y)\bigr)_v=\overline{\bigl(T(w):\nabla(y)\bigr)_v}=v^{-\ell(y)+\ell(w)}P_{w_0y,w_0w}(v^{2})=v^{-\ell(y)+\ell(w)}Q_{w,y}(v^2),
$$
which implies that $\varphi(\hucH_w)=[T(w)]$ by (\ref{uch2}).
\end{proof}

\begin{rem} We note that $\varphi(\nabla(w))$ is in general not equal to $T_{w^{-1}}^{-1}$. In particular, $\varphi$ does not intertwine the duality functor and the bar involution.
\end{rem}

Let $\mathcal{P}$ be the category of graded projective endofunctors of $\O_0^\Z$ and $[\mathcal{P}]$ be its Grothendieck group. For each $w\in W$ and $k\in\Z$, we define $v^k[\theta_w]:=[\theta_w\<k\>]$. By linearity we get a $\Z[v,v^{-1}]$-module structure on $[\mathcal{P}]$.

\begin{lem}\text{(\cite[Theorem 7.11]{Ma2})} With the notations as above, the map which sends $\theta_w$ to $\uH_w$ for each $w\in W$ can be extended uniquely to an anti-isomorphism of $\Z[v,v^{-1}]$-algebras between
$[\mathcal{P}]$ and the Hecke algebra $\HH(W)$.
\end{lem}
In other words, the above lemmas says that the category of the graded projective endofunctors of $\O_0$ gives a categorification of the right regular $\HH(W)$-module. The next lemma gives a second version of isomorphism between $\HH(W)$ and $\bigl[\O_0^\Z\bigr]$ which seems not explicitly stated anywhere in the literature.

\begin{lem}\label{keylem02} There is a unique $\Z[v,v^{-1}]$-module isomorphism: $\psi: \HH(W)\cong\bigl[\O_0^\Z\bigr]$ such that \begin{equation}\label{3maps}
\psi(H_{w})=[\nabla(w_0w)],\quad \psi(\uH_{w})=[T(w_0w)],\quad \psi(\hucH_{w})=[I(w_0w)],   \,\forall\,w\in W,
\end{equation}
Moreover, we have \begin{equation}\label{2maps} \psi(H_{w^{-1}}^{-1})=[\Delta(w_0w)],\quad \psi(\ucH_{w})=[L(w_0w)], \forall\,w\in W. \end{equation}
\end{lem}

\begin{proof} The graded Ringel self-duality functor (\ref{RingelDuality}) induces an isomorphism $\bigl[\O_0^\Z\bigr]\cong \bigl[\O_0^\Z\bigr]$ such that $$
[P(w)]\mapsto [T(w_0w)],\,\, [\Delta(w)]\mapsto [\nabla(w_0w)],\,\,[T(w)]\mapsto [I(w_0w)] .
$$
This gives to a unique $\Z[v,v^{-1}]$-module isomorphism: $\psi: \HH(W)\cong\bigl[\O_0^\Z\bigr]$ such that (\ref{3maps}) holds. In particular, $$
 [T(w_0w)]  = [\nabla(w_0w)]+\sum_{y<w}v^{\ell(w)-\ell(y)}P_{y,w}(v^{-2})[\nabla(w_0y)].
$$

Now, $\overline{\uH_w}=\uH_w$, $T(w_0w)^{\circledast}\cong T(w_0w)$, we have $$\begin{aligned}
& [T(w_0w)]  = [T(w_0w)^{\circledast}] = [\Delta(w_0w)]+\sum_{y<w}v^{\ell(y)-\ell(w)}P_{y,w}(v^2)[\Delta(w_0y)],\\
& \uH_w = H_{w^{-1}}^{-1}+\sum_{y<w}v^{\ell(y)-\ell(w)}P_{y,w}(v^2)H_{y^{-1}}^{-1}.
\end{aligned}
$$
By an induction on the Bruhat order ``$<$'', we can deduce that $\psi(H_{w^{-1}}^{-1})=[\Delta(w_0w)]$.

Finally, since $$
[\nabla(w_0y):L(w_0w)]=v^{\ell(w)-\ell(y)}P_{w_0y,w_0w}(v^2)=v^{\ell(w)-\ell(y)}Q_{w,y}(v^2).
$$
It follows from \cite[Theorem 3.11.4]{BGS} and \cite[Theorem 3.1]{KL} that $$
[L(w_0w)]=[\nabla(w_0w)]+\sum_{y<w}(-v)^{-\ell(w)+\ell(y)}P_{y,w}(v^2)[\nabla(w_0y)] .
$$
On the other hand, by (\ref{uch}), we have $$
\ucH_w=H_w+\sum_{y<w}(-v)^{-\ell(w)+\ell(y)}P_{y,w}(v^2)H_y .
$$
Comparing the above two equalities and use an induction on the Bruhat order ``$<$'', we can deduce that $\psi(\ucH_w)=[L(w_0w)]$.
\end{proof}


\begin{cor}\label{vermacor} Let $x\in W$. Suppose that \begin{equation}\label{verma}
H_{(w_0x)^{-1}}^{-1}=H_{w_0x}+\sum_{x<y\in W}r_{y,x}(v)H_{w_0y},
\end{equation}
where $r_{y,x}(v)\in\Z[v,v^{-1}]$ for each $y\in W$. Then in the Grothendieck group $[A\gmod]$, we have $$
[\Delta(x)]=[\nabla(x)]+\sum_{x<y\in W}r_{y,x}(v)[\nabla(y)]. $$
Moreover, we have $r_{y,x}(v)=r_{y^{-1},x^{-1}}(v)$ for any $x,y\in W$.
\end{cor}

\begin{proof} The first part of the corollary follows from Lemma \ref{keylem02}. It remains to show $r_{y,x}(v)=r_{y^{-1},x^{-1}}(v)$ for any $y\in W$.

For any $x,w\in W$ with $x\leq w$, we denote by $R_{x,w}(q)$ the $R$-polynomial as defined in \cite[(2.0.a)]{KL}, where $q:=v^{-2}$. Then by \cite[Lemma 2.1(i)]{KL} we have $$
r_{y,x}(v)=(-1)^{\ell(x)+\ell(y)}v^{\ell(y)-\ell(x)}R_{w_0y,w_0x}(v^{-2}) .
$$
Applying \cite[Lemma 2.1(iv)]{KL} and the fact $R_{x,w}(q)=R_{x^{-1},w^{-1}}(q)$ (which can be proved by applying the anti-isomorphism $\ast$), we can deduce that $$\begin{aligned}
r_{y,x}(v)&=(-1)^{\ell(x)+\ell(y)}v^{\ell(y)-\ell(x)}R_{w_0y,w_0x}(v^{-2})=(-1)^{\ell(x)+\ell(y)}v^{\ell(y)-\ell(x)}R_{x,y}(v^{-2})\\
&=(-1)^{\ell(x)+\ell(y)}v^{\ell(y)-\ell(x)}R_{x^{-1},y^{-1}}(v^{-2})=(-1)^{\ell(x^{-1})+\ell(y^{-1})}v^{\ell(y^{-1})-\ell(x^{-1})}R_{w_0y^{-1},w_0x^{-1}}(v^{-2})=r_{y^{-1},x^{-1}}(v).
\end{aligned}
$$
This completes the proof of the lemma.
\end{proof}
\bigskip

\section{A categorical action of Hecke algebra on derived category via derived twisting functors}

The purpose of this section is to show that there is a categorical action of the Hecke algebra $\HH(W)$ on the derived category $D^b(\O_0^{\Z})$ via derived twisting functors.

\begin{lem}\text{(\cite[(2.3), Theorem 2.3]{AS})}\label{Tsaction1} Let $x\in W$ and $s\in S$. There are the following isomorphisms in $\O_0^\Z$: $$
T_s\nabla(x)\cong\begin{cases}\nabla(x)\<-1\>, &\text{if $x<sx$;}\\ \nabla(sx), &\text{if $x>sx$.}\end{cases}
$$
Moreover, if $sx>x$, then $T_s\Delta(x)\cong\Delta(sx)$.
\end{lem}

\begin{proof} If we forget the $\Z$-grading, then the lemma is just \cite[(2.3), Theorem 2.3]{AS}. Suppose that $sx<x$. Then by \cite[Appendix, Proposition 7]{MO}, $T_s\nabla(x)\cong\nabla(sx)$.

Now assume $sx>x$. Then $sw_0\geq x$ and hence $\ell(w_0)-1\geq\ell(x)$. Since $$\begin{aligned}
&\quad \hom_{\O_0^\Z}\bigl(T_s\nabla(x), L(w_0)\<-\ell(w_0)+\ell(x)-1\>\bigr)\cong\hom_{\O_0^\Z}\bigl(\nabla(x), G_sL(w_0)\<-\ell(w_0)+\ell(x)-1\>\bigr)\\
&\cong \hom_{\O_0^\Z}\bigl(\nabla(x), (T_sL(w_0))^{\circledast}\<-\ell(w_0)+\ell(x)-1\>\bigr)\cong\hom_{\O_0^\Z}\bigl(\nabla(x), (T_s\nabla(w_0))^{\circledast}\<-\ell(w_0)+\ell(x)-1\>\bigr)\\
&\cong\hom_{\O_0^\Z}\bigl(\nabla(x), \nabla(sw_0)^\circledast\<-\ell(w_0)+\ell(x)-1\>\bigr)\cong\hom_{\O_0^\Z}\bigl(\nabla(x), \Delta(sw_0)\<-\ell(w_0)+\ell(x)-1\>\bigr).
\end{aligned}
$$
Forgetting the $\Z$-grading, we know that $$
\Hom_{A}\bigl(\nabla(x), \Delta(sw_0)\bigr)\cong\Hom_{A}\bigl(T_s\nabla(x), L(w_0)\bigr)\cong\Hom_{A}\bigl(\nabla(x), L(w_0)\bigr)\cong\mathbb{C} .
$$
On the other hand, $\nabla(x)$ has simple head $L(w_0)\<-\ell(w_0)+\ell(x)\>$, while $\Delta(sw_0)\<-\ell(w_0)+\ell(x)-1\>$ has simple socle $L(w_0)\<-\ell(w_0)+\ell(x)\>$. It follows that $$
\hom_{\O_0^\Z}\bigl(T_s\nabla(x), L(w_0)\<-\ell(w_0)+\ell(x)-1\>\bigr)\cong\hom_{\O_0^\Z}\bigl(\nabla(x), \Delta(sw_0)\<-\ell(w_0)+\ell(x)-1\>\bigr)\cong\C .
$$
This proves that $T_s\nabla(x)\cong\nabla(x)\<-1\>$.

Now as $s(sx)<sx$, we have $$\begin{aligned}
&\quad\,\hom_{\O_0^\Z}\bigl(T_s\Delta(x),\Delta(sx)\bigr)\cong \hom_{\O_0^\Z}\bigl(\Delta(x), G_s\Delta(sx)\bigr)\cong\hom_{\O_0^\Z}\bigl(\Delta(x), (T_s\nabla(sx))^\circledast\bigr)\cong\hom_{\O_0^\Z}\bigl(\Delta(x), (\nabla(x))^\circledast\bigr)\\
&\cong\hom_{\O_0^\Z}\bigl(\Delta(x), \Delta(x)\bigr)\cong\C ,
\end{aligned}
$$
it follows that $T_s\Delta(x)\cong\Delta(sx)$ in this case.
\end{proof}

\begin{lem}\text{(\cite[Theorem 2.3]{AS})}\label{Tsaction2} Let $x\in W$ and $s\in S$. There are the following isomorphisms in $\O_0^\Z$: $$
(\mL_1T_s)\nabla(x)\cong\begin{cases}K_{x,sx}\<1\>, &\text{if $x<sx$;}\\ 0, &\text{if $x>sx$,}\end{cases}
$$
where $K_{x,sx}$ denotes the kernel of the (unique up to a scalar) nontrivial surjective homomorphism $\nabla(x)\twoheadrightarrow\nabla(sx)\<-1\>$ in the case $x<sx$.
\end{lem}

\begin{proof} If we forget the $\Z$-grading, then the lemma is just the second part of \cite[Theorem 2.3]{AS}. In the graded setting, we shall prove the lemma by translating the argument in the proof of \cite[Theorem 2.3]{AS} into the $\Z$-graded setting.

We use induction on $\ell(x)$. If $x=w_0$, then by \cite[Theorem 2.3]{AS}, we know that $(\mL_1T_s)\nabla(x)=0$. Now assume $x\neq w_0$. We choose a simple reflection $t\in S$ such that $xt>x$. Applying \cite[(5.3),(5.6)]{Stro2}, we have the following short exact sequence in $\O_0^\Z$: $$
0\rightarrow\nabla(xt)\<1\>\overset{f}{\rightarrow}\theta_t\nabla(xt)\rightarrow\nabla(x)\rightarrow 0 .
$$
Applying the functor $T_s$ and using Lemma \ref{lemacylic1} and Corollary \ref{mltwisted}, we get the following long exact sequence: \begin{equation}\label{les1}
0\rightarrow(\mL_1T_s)\nabla(xt)\<1\>\rightarrow \theta_t(\mL_1T_s)\nabla(xt)\rightarrow(\mL_1T_s)\nabla(x)\rightarrow T_s\nabla(xt)\<1\>\rightarrow T_s\theta_t\nabla(xt)\rightarrow T_s\nabla(x)\rightarrow 0 .
\end{equation}

\smallskip
{\it Case 1.} $sxt<xt$. Applying induction hypothesis, we see $(\mL_1T_s)\nabla(xt)=0$.
Thus (\ref{les1}) becomes the following sequence  \begin{equation}\label{les2}
0\rightarrow (\mL_1T_s)\nabla(x)\rightarrow T_s\nabla(xt)\<1\>\overset{T_sf}{\rightarrow} T_s\theta_t\nabla(xt)\rightarrow T_s\nabla(x)\rightarrow 0 .
\end{equation}
If $sxt>sx$, then $sx<x$. By the proof of \cite[Theorem 2.3]{AS}, we know that $(\mL_1T_s)\nabla(x)=0$. Henceforth, we assume that $sxt<sx$ and hence $sx>x$. By Exchange Condition, we can deduce that $sxt=x$. Note that  by the proof of \cite[Theorem 2.3]{AS}, $T_sf$ is equal to the composite of the surjection $$
T_s\nabla(xt)\<1\>\cong\nabla(sxt)\<1\>\twoheadrightarrow\nabla(sx)$$ and the adjunction morphism (\cite[Theorem 5.3]{Stro2}) $\nabla(sx)\hookrightarrow \theta_t\nabla(sxt)\cong T_s\theta_t\nabla(xt)$. Hence $\mL_1T_s\nabla(x)\cong K_{sxt,sx}\<1\>=K_{x,sx}\<1\>$. This proves the lemma in the case $sxt<xt$.

\smallskip
{\it Case 2.} $sxt>xt$. By Lemma \ref{Tsaction1}, $sx>x$ implies that $T_s\nabla(x)\cong\nabla(x)\<-1\>$. By Lemma \ref{Tsaction1}, $T_s\nabla(xt)\cong\nabla(xt)\<-1\>$. As in the proof of \cite[Theorem 2.3]{AS}, we have that the morphism $T_sf$ is injective with cokernel $\nabla(x)\<-1\>$. Thus, applying induction hypothesis, (\ref{les1}) and (\cite[Theorem 5.3, Corollary 5.5]{Stro2}), we get the following exact sequence \begin{equation}\label{KKT1}
0\rightarrow K_{xt,sxt}\<2\>\overset{a}{\rightarrow}\theta_tK_{xt,sxt}\<1\>\rightarrow\mL_1 T_s\nabla(x)\rightarrow 0 ,
\end{equation}
where $a$ is the restriction of the adjunction morphism $f$. As  the proof of \cite[Theorem 2.3]{AS}, we have the following commutative diagram with exact rows and the surjection $p$:  $$
\begin{CD}
    \nabla(xt)\<2\> @> {\rm{adj}}>> \theta_t\nabla(xt)\<1\> @>  >> \nabla(x)\<1\>  \\
    @V  V p V @VV \theta_t p V  @VV  V\\
    \nabla(sxt)\<1\> @>>{\rm{adj}}> \theta_t\nabla(sxt) @>>  > \nabla(sx)
\end{CD}. $$
It follows that we have the following exact sequence: $$
0\rightarrow K_{xt,sxt}\<2\>\rightarrow\theta_t K_{xt,sxt}\<1\>\rightarrow K_{x,sx}\<1\>\rightarrow 0 .
$$
By comparing the above exact sequence with (\ref{KKT1}), we get that $\mL_1 T_s\nabla(x)\cong K_{x,sx}\<1\>$. This completes the proof of the lemma.
\end{proof}

\noindent
{\textbf{Proof of Theorem \ref{mainthm1}:}} We first show that for any $s\in S$, $(\mL T_s-v^{-1})(\mL T_s+v)=0$ on the Grothendieck group of $D^b(\O_0^\Z)$. It suffices to show that for any $x\in W$, \begin{equation}\label{eq1}
(\mL T_s-v^{-1})(\mL T_s+v)[\nabla(x)]=0 .
\end{equation}

Suppose $sx>x$. Then by Lemmas \ref{Tsaction1} and \ref{Tsaction2}, we have $$
(\mL T_s+v)[\nabla(x)]=[\nabla(x)\<-1\>]-[K_{x,sx}\<1\>]+v[\nabla(x)]=v^{-1}[\nabla(x)]+[\nabla(sx)].
$$
Thus, $$\begin{aligned}
&\quad (\mL T_s-v^{-1})(\mL T_s+v)[\nabla(x)]\\
&=v^{-1}[\mL T_s\nabla(x)]-v^{-2}[\nabla(x)]+[\mL T_s\nabla(sx)]-v^{-1}[\nabla(sx)]\\
&=v^{-2}[\nabla(x)]-([\nabla(x)]-v^{-1}[\nabla(sx)])-v^{-2}[\nabla(x)]+[\nabla(x)]-v^{-1}[\nabla(sx)]\\
&=0.
\end{aligned}
$$
Now suppose that $sx<x$. Then by Lemmas \ref{Tsaction1} and \ref{Tsaction2}, we have that $$
(\mL T_s+v)[\nabla(x)]=[\nabla(sx)]+v[\nabla(x)].
$$
Thus, $$\begin{aligned}
&\quad (\mL T_s-v^{-1})(\mL T_s+v)[\nabla(x)]\\
&=[\mL T_s\nabla(sx)]-v^{-1}[\nabla(sx)]+v[(\mL T_s\nabla(x)]-[\nabla(x)]\\
&=v^{-1}[\nabla(sx)]-v([\nabla(sx)]-v^{-1}[\nabla(x)])-v^{-1}[\nabla(sx)]+v[\nabla(sx)]-[\nabla(x)]\\
&=0.
\end{aligned}
$$
This completes the proof of (\ref{eq1}).

Second, we want to show that for any $u,w\in W$ with $\ell(uw)=\ell(u)+\ell(w)$, $\mL T_u\mL T_w=\mL T_{uw}$ on the Grothendieck group of $D^b(\O_0^\Z)$. Using Lemma \ref{lemacylic1}, it suffices to show that for any $x\in W$,
\begin{equation}\label{eq2}
[T_u T_w\Delta(x)]=[T_{uw}(\Delta(x))].
\end{equation}
However, this follows from (\ref{braid1}). Now to complete the proof of the first part of the theorem, it remains to show that
$[(\mL T_s)\nabla(x)]=H_{w_0x^{-1}}H_s,\forall\,s\in S, x\in W$.

Let $s\in S$ and $x\in W$. Suppose $sx<x$. Then $x^{-1}s<x^{-1}$ and hence $(w_0x^{-1})s>w_0x^{-1}$. Applying Lemma \ref{Tsaction1}, we get that $T_s\nabla(x)\cong\nabla(sx)$. On the other hand, we have $$
H_{w_0x^{-1}}H_s=H_{w_0x^{-1}s}.
$$
Hence $[(\mL T_s)\nabla(x)]=[\nabla(sx)]=H_{w_0x^{-1}s}=H_{w_0x^{-1}}H_s$.

Now suppose that $sx>x$. In this case, $x^{-1}s>x^{-1}$ and hence $(w_0x^{-1})s<w_0x^{-1}$. Applying Lemma \ref{Tsaction2} we can deduce that $$\begin{aligned}
\bigl[{\mL}T_s\nabla(x)\bigr]&=[T_s\nabla(x)]-[\mL_1 T_s\nabla(x)]=[\nabla(x)\<-1\>]-[K_{x,sx}\<1\>]\\
&=v^{-1}[\nabla(x)]-\bigl([\nabla(x)\<1\>]-[\nabla(sx)]\bigr)=v^{-1}[\nabla(x)]-\bigl(v[\nabla(x)]-[\nabla(sx)]\bigr)\\
&=[\nabla(sx)]+(v^{-1}-v)[\nabla(x)] .
\end{aligned}
$$
On the other hand, the assumption that $sx>x$ implies that $$
H_{w_0x^{-1}}H_s=(H_{w_0x^{-1}s}H_s)H_s=H_{w_0x^{-1}s}H_s^2=H_{w_0x^{-1}s}((v^{-1}-v)H_s+1)=(v^{-1}-v)H_{w_0x^{-1}}+H_{w_0x^{-1}s}.
$$
This proves that $[(\mL T_s)\nabla(x)]=H_{w_0x^{-1}}H_s$.

By \cite[Theorem 3.11.4]{BGS} and \cite[Theorem 3.1]{KL}, we have \begin{equation}\label{Lnablax}
[L(x)]=[\nabla(x)]+\sum_{y>x}(-v)^{\ell(x)-\ell(y)}P_{w_0y,w_0x}(v^{2})[\nabla(y)].
\end{equation}
Applying Lemma \ref{keylem02}, we get that \begin{equation}\label{Lnabla1}
\ucH_{w_0x}=H_{w_0x}+\sum_{y>x}(-v)^{\ell(x)-\ell(y)}P_{w_0y,w_0x}(v^{2})H_{w_0y}.
\end{equation}
Applying \cite[Corollaries 4.3, 4.4]{Bre}, we have $$
P_{w_0y,w_0x}(v^{2})=P_{y^{-1}w_0,x^{-1}w_0}(v^{2})=P_{w_0y^{-1},w_0x^{-1}}(v^{2}). $$
Then we get that $$
[L(x)]=[\nabla(x)]+\sum_{y>x}(-v)^{\ell(x)-\ell(y)}P_{w_0y^{-1},w_0x^{-1}}(v^{2})[\nabla(y)]. $$
Hence $$\begin{aligned}
\rho([L(x)])&=H_{w_0x^{-1}}+\sum_{y^{-1}>x^{-1}}(-v)^{\ell(x)-\ell(y)}P_{w_0y^{-1},w_0x^{-1}}(v^{2})H_{w_0y^{-1}}\\
&=H_{w_0x^{-1}}+\sum_{w_0y^{-1}<w_0x^{-1}}(-v)^{\ell(w_0y^{-1})-\ell(w_0x^{-1})}P_{w_0y^{-1},w_0x^{-1}}(v^{2})H_{w_0y^{-1}}\\
&=\ucH_{w_0x^{-1}} .
\end{aligned}$$

Now assume $x$ is an involution. Applying (\ref{verma}), we have that \begin{equation}\label{w01}
H_{(w_0x)^{-1}}^{-1}=H_{w_0x}+\sum_{x<y\in W}r_{y,x}(v)H_{w_0y},
\end{equation}
where $r_{y,x}(v)\in\Z[v,v^{-1}]$ for each $y\in W$. By Corollary \ref{vermacor}, we get that in the Grothendieck group $[A\gmod]$,  \begin{equation}\label{w02}
[\Delta(x)]=[\nabla(x)]+\sum_{x<y\in W}r_{y,x}(v)[\nabla(y)] .
\end{equation}
By Corollary \ref{vermacor}, we have $r_{y,x}(v)=r_{y^{-1},x^{-1}}(v)$. Now assume $x$ is an involution. It follows that $$\begin{aligned}
\rho\bigl([\Delta(x)]\bigr)&=\rho([\nabla(x)])+\sum_{x<y\in W}r_{y,x}(v)\rho([\nabla(y)])
=\rho([\nabla(x)])+\sum_{x<y\in W}r_{y^{-1},x^{-1}}(v)\rho([\nabla(y)])\\
&=\rho([\nabla(x)])+\sum_{x<y\in W}r_{y,x}(v)\rho([\nabla(y^{-1})])\\
&=H_{w_0x^{-1}}+\sum_{x<y\in W}r_{y,x}(v)H_{w_0y}\\
&=H_{w_0x}+\sum_{x<y\in W}r_{y,x}(v)H_{w_0y}\quad \text{\rm (as $x=x^{-1}$)}\\
&=H_{(w_0x)^{-1}}^{-1}=H_{xw_0}^{-1}
\end{aligned}
$$
This completes the proof of Theorem \ref{mainthm1}.
\hfill\qed
\medskip

\begin{cor} Let $x\in W$. Then in the Grothendieck group of $\O_0^\Z$, we have $$
[\nabla(x)]-[L(x)]=[\nabla(x^{-1})]-[L(x^{-1}].
$$
\end{cor}

\begin{proof} This follows from (\ref{Lnablax}) and (\ref{Pmu}).
\end{proof}

\begin{rem} By a similar argument, one can show that there is a $\Z[v,v^{-1}]$-module isomorphism $\rho'$ from the Grothendieck group of $D^b(\O_0^\Z)$ onto $\HH(W)$ defined by $$
[\nabla(x)\<k\>]\mapsto v^k H_{xw_0},\quad \forall\,x\in W, k\in\Z ,
$$
and the derived twisting functors $\mL T_w$ gives rise to a categorical action of the Iwahori-Hecke algebra $\HH(W)$ on $D^b(\O_0^\Z)$ such that $$
\rho'\bigl([(\mL T_x)\nabla(y)]\bigr)=H_xH_{yw_0},\quad\forall\,x,y\in W .
$$
\end{rem}

\begin{lem}\label{TsPx} Let $s\in S$ and $x\in W$ with $sx<x$. Then we have $T_sP(x)\cong P(x)\<-1\>$ and $\mathcal{L}_1T_sL(sx)\cong L(sx)\<1\>$.
\end{lem}  \begin{proof} Since $sx<x$, we have $[\Delta(sx):L(x)\<1\>]=1$. Therefore, $$
\dim\hom\bigl(T_sP(x),\Delta(x)\<-1\>\bigr)=\dim\hom\bigl(T_sP(x),T_s\Delta(sx)\<-1\>\bigr)=\dim\hom\bigl(P(x),\Delta(sx)\<-1\>\bigr)=1.
$$
On the other hand, by \cite[Proposition 5.3]{AS} we have that $T_sP(x)$ is isomorphic to $P(x)$ upon forgetting their $\Z$-gradings. It follows that $T_sP(x)\cong P(x)\<-1\>$.

Finally, since $\mathcal{L}_iT_s=0$ for any $i>1$ (Lemma \ref{lemacylic1}), we have a natural embedding $\mathcal{L}_1T_sL(sx)\hookrightarrow\mathcal{L}_1T_s\nabla(sx)$. By Lemma \ref{Tsaction2},
we have $\mathcal{L}_1T_s\nabla(sx)\cong K_{sx,x}\<1\>$. Combing this with \cite[Theorem 6.1]{AS} in the ungraded setting, we can deduce that $\mathcal{L}_1T_sL(sx)\cong L(sx)\<1\>$.
\end{proof}

\noindent
{\textbf{Proof of Theorem \ref{mainthm2}:}} Since $T_s$ is right exact, the natural degree $0$ surjection $\phi: P(x)\twoheadrightarrow L(x)$ induces a degree $0$ surjection
$T_s\phi: T_sP(x)\twoheadrightarrow T_sL(x)$. By assumption, $sx<x$, hence $T_sL(x)\neq 0$ by \cite[Proposition 5.1]{AS}. Hence $T_s\phi\neq 0$. Applying Lemma \ref{TsPx}, we have $$
\hom\bigl(P(x)\<-1\>,T_sL(x)\bigr)\cong \hom\bigl(T_sP(x),T_sL(x)\bigr).
$$
It follows that there exists a nonzero degree $0$ map from $P(x)\<-1\>$ to $T_sL(x)$. On the other hand, upon forgetting the $\Z$-grading, we know that $L(x)$ occurs as a composition factor in $T_sL(x)$ with multiplicity $1$ by
\cite[Theorem 6.3]{AS}, and $[\hd T_sL(x):L(x)]=1$. It follows that $L(x)\<-1\>$ is the unique simple head of $T_sL(x)$ and $[T_sL(x):L(x)]_v=v^{-1}$.

Now we show that $L(sx)$ appears as a graded composition factor in $T_sL(x)$. Using Lemma \ref{TsPx} and a $\Z$-graded version of the argument used in \cite[Theorem 6.3]{AS}, we can deduce that $$\begin{aligned}
&\quad\,\dim\hom_A(L(sx),T_sL(x))=\dim\hom_A(\mathcal{L}_1T_sL(sx)\<-1\>,T_sL(x))\\
&=\dim\hom_A((\mathcal{L}_1T_s)L(sx),T_sL(x)\<1\>)\\
&=\dim\hom_{D^b(A)}((\mathcal{L}T_s)L(sx),(\mathcal{L}T_s)L(x)[1]\<1\>)\\
&=\dim\hom_{D^b(A)}(L(sx),L(x)[1]\<1\>)=\dim \ext_A^1(L(sx),L(x)\<1\>)=\mu(x,sx)=1 .\end{aligned}
$$
Similarly, we have $\dim\Hom_A(L(sx),T_sL(x))=1$. It follows that $[\soc T_sL(x):L(sx)]_v=1$.

Note that $sx<x$ implies that $x^{-1}s<x^{-1}$. Applying Theorem \ref{mainthm1} and (\ref{2Hs}), we get that $$\begin{aligned}
{}\rho([T_sL(x)])&=\ucH_{w_0x^{-1}}H_s=\ucH_{w_0x^{-1}}\ucH_s+v^{-1}\ucH_{w_0x^{-1}}\\
&=v^{-1}\ucH_{w_0x^{-1}}+\ucH_{w_0x^{-1}s}+\sum_{\substack{y\in W\\ w_0y^{-1}s<w_0y^{-1}<w_0x^{-1}}}\mu(w_0y^{_1},w_0x^{-1})\ucH_{w_0y^{-1}}\\
&=v^{-1}\ucH_{w_0x^{-1}}+\ucH_{w_0x^{-1}s}+\sum_{\substack{y\in W\\ sy>y>x}}\mu(x,y)\ucH_{w_0y^{-1}}\\
&=v^{-1}[L(x)]+[L(sx)]+\sum_{\substack{y\in W\\ sy>y>x}}\mu(x,y)[L(y)].
\end{aligned} $$

On the other hand, applying Theorem \ref{mainthm1}, we get that $$\begin{aligned}
{}\rho([T_sL(x)])&=\ucH_{w_0x^{-1}}H_s\\
&=\sum_{y\geq x}(-v)^{\ell(x)-\ell(y)}P_{w_0y^{-1},w_0x^{-1}}(v^2)H_{w_0y^{-1}}H_s\\
&=\sum_{\substack{y\geq x\\ y^{-1}s<y^{-1}}}(-v)^{\ell(x)-\ell(y)}P_{w_0y^{-1},w_0x^{-1}}(v^2)H_{w_0y^{-1}}H_s+\sum_{\substack{y\geq x\\ y^{-1}s>y^{-1}}}(-v)^{\ell(x)-\ell(y)}P_{w_0y^{-1},w_0x^{-1}}(v^2)H_{w_0y^{-1}}H_s\\
&=\sum_{\substack{y\geq x\\ x^{-1}\nleq y^{-1}s<y^{-1}}}(-v)^{\ell(x)-\ell(y)}P_{w_0y^{-1},w_0x^{-1}}(v^2)H_{w_0y^{-1}s}+\sum_{\substack{y\geq x\\ x^{-1}\leq y^{-1}s<y^{-1}}}(-v)^{\ell(x)-\ell(y)}P_{w_0y^{-1},w_0x^{-1}}(v^2)H_{w_0y^{-1}s}\\
&\qquad +\sum_{\substack{y\geq x\\ y^{-1}s>y^{-1}}}(-v)^{\ell(x)-\ell(y)}P_{w_0y^{-1},w_0x^{-1}}(v^2)H_{w_0y^{-1}}H_s\\
&=\sum_{\substack{y\geq x\\ x^{-1}\nleq y^{-1}s<y^{-1}}}(-v)^{\ell(x)-\ell(y)}P_{w_0y^{-1},w_0x^{-1}}(v^2)H_{w_0y^{-1}s}+\sum_{\substack{y\geq x\\  y^{-1}s>y^{-1}}}(-v)^{\ell(x)-\ell(y)-1}P_{w_0y^{-1}s,w_0x^{-1}}(v^2)H_{w_0y^{-1}}\\
&\qquad +\sum_{\substack{y\geq x\\ y^{-1}s>y^{-1}}}(-v)^{\ell(x)-\ell(y)}P_{w_0y^{-1},w_0x^{-1}}(v^2)\bigl((v^{-1}-v)H_{w_0y^{-1}}+H_{w_0y^{-1}s}\bigr)\\
\end{aligned}
$$
Applying \cite[(2.3.g)]{KL} and \cite[Corollary 4.4]{Bre}, we see that for any $y,w\in W$ with $y<w, ys<y, ws>w$, $$
P_{y,w}(v^2)=P_{ys,w}(v^2) .
$$
Therefore, $$\begin{aligned}
{}\rho([T_sL(x)])&=\sum_{\substack{y\geq x\\ x^{-1}\nleq y^{-1}s<y^{-1}}}(-v)^{\ell(x)-\ell(y)}P_{w_0y^{-1},w_0x^{-1}}(v^2)H_{w_0y^{-1}s}+\sum_{\substack{y\geq x\\ y^{-1}s>y^{-1}}}(-v)^{\ell(x)-\ell(y)}P_{w_0y^{-1},w_0x^{-1}}(v^2)\bigl(-vH_{w_0y^{-1}}+H_{w_0y^{-1}s}\bigr)\\
&=\sum_{\substack{y\geq x\\ x^{-1}\nleq y^{-1}s<y^{-1}}}(-v)^{\ell(x)-\ell(y)}P_{w_0y^{-1},w_0x^{-1}}(v^2)H_{w_0y^{-1}s}+\sum_{\substack{y\geq x\\ sy>y}}(-v)^{\ell(x)-\ell(y)+1}\biggl(P_{yw_0,xw_0}(v^2)(H_{w_0y^{-1}}-v^{-1}H_{w_0y^{-1}s})\biggr).
\end{aligned}
$$
It follows that
$$
[T_sL(x)]=\sum_{\substack{y\geq x\\ x\nleq sy<y}}(-v)^{\ell(x)-\ell(y)}P_{w_0y^{-1},w_0x^{-1}}(v^2)[\nabla(sy)]+\sum_{\substack{y\geq x\\ sy>y}}(-v)^{\ell(x)-\ell(y)+1}P_{yw_0,xw_0}(v^2)\bigl([\nabla(y)]-v^{-1}[\nabla(sy)]\bigr) .
$$
This completes the proof of Theorem \ref{mainthm2}.
\hfill\qed
\medskip

The following corollary was first proved in \cite[Theorem 6.3, Theorem 7.8]{AS} in the ungraded case, we generalize it to the $\Z$-graded setting.

\begin{cor}\text{(\cite{AS})}\label{TsLx1} Let $x\in W$ and $s\in S$ with $sx<x$. Then the Loewy length of $T_sL(x)$ is equal to $2$, $\hd T_sL(x)\cong L(x)\<-1\>$ and $$
\soc T_sL(x)=L(sx)\oplus\bigoplus_{\substack{y\in W\\ sy>y>x}}L(y)^{\oplus\mu(x,y)}.
$$
\end{cor}

\begin{proof} This follows from Lemma \ref{parity1} and Theorem \ref{mainthm2}.
\end{proof}

Let $s\in S$ and $x\in W$. Following \cite{AS}, we call $L(x)$ is $s$-finite if $sx>x$, and is $s$-free if $sx<x$. By \cite[Corolary 5.8]{AS}, $T_sM=0$ if and only if $M$ is $s$-finite (i.e, every composition factor of $M$ is $s$-finite). Let $Z_s$ and $\hat{Z}_s$ be the graded Zuckerman functor and the dual graded Zuckerman functor associated to $s$, see \cite[\S6.1]{Ma2}, \cite[(2.2)]{Bac} and \cite[\S3.1]{HX23}. By \cite{EW}, $\mathcal{L}_2Z_s$ is isomorphic to $\hat{Z}_s$ upon forgetting the $\Z$-grading.

By \cite{AS} and \cite{MaStro07}, we know that $\mL_2 Z_s\cong\hat{Z}_s$ upon forgetting the $\Z$-grading. The following lemma explicitly determine the degree shift when acting on simple modules in the $\Z$-graded lift setting.

\begin{lem}\label{keylem3} Let $x\in W$ and $s\in S$. Then we have $\mL_2 Z_sL(x)\cong\hat{Z}_sL(x)\<1\>$. If $sx>x$ then $\mathcal{L}_1Z_sL(x)=0$.
\end{lem}

\begin{proof} If $sx<x$, then as in the ungraded case, $\mL_2 Z_s L(x)=0=\hat{Z}_sL(x)\<1\>$. Henceforth, we assume $sx>x$.

Let $K(x):=\kor p$, $p:\Delta(x)\rightarrow L(x)$ is the canonical surjection. We have the following exact sequence: \begin{equation}
\mathcal{L}_1Z_sK(x)\rightarrow\mathcal{L}_1Z_s\Delta(x)\rightarrow \mathcal{L}_1Z_sL(x)\rightarrow Z_sK(x)\overset{Z_s(\iota)}{\rightarrow} Z_s\Delta(x){\rightarrow} Z_sL(x)\rightarrow 0,
\end{equation}
where $\iota: K(x)\rightarrow\Delta(x)$ is the natural embedding.

Since $sx>x$, applying \cite[Claim 3.2]{MaStro07}, we have $\mathcal{L}_1Z_s\Delta(x)=0$. Note that if $N\subseteq K(x)$ is a submodule such that $K(x)/N$ has only $s$-finite composition factors, then $\Delta(x)/N$ has only $s$-finite composition factors as well. It follows that the map $Z_s(\iota)$ is injective, which implies that the natural map $\mathcal{L}_1Z_sL(x)\rightarrow Z_sK(x)$ in (\ref{exact2}) is a zero map, and hence forces $\mathcal{L}_1Z_sL(x)=0$.

Note that $\mL_2Z_s\Delta(x)=\hat{Z}_s\Delta(x)=0$. It follows that there is the following exact sequence: $$
\mL_2Z_s\Delta(x)=0\rightarrow\mL_2Z_sL(x)\rightarrow\mL_1Z_sK(x)\rightarrow\mL_1Z_s\Delta(x)\rightarrow\mL_1Z_sL(x)=0 .
$$
It follows that \begin{equation}\label{eqa11}
[\mL_2Z_sL(x)]=[\mL_1Z_sK(x)]-[\mL_1Z_s\Delta(x)] .
\end{equation}

Since $G_s$ is left exact, we have an exact sequence $0\rightarrow G_s K(x)\rightarrow G_s\Delta(x)\rightarrow G_sL(x)$. Now $sx>x$ implies that $G_sL(x)=0$. Hence the embedding $G_s K(x)\hookrightarrow G_s\Delta(x)$ is an isomorphism. That is, $G_sK(x)\cong G_s\Delta(x)$. As a result, $T_sG_sK(x)\cong T_sG_s\Delta(x)$. On the other hand, by \cite[5.7,5.9]{AS}, \cite{MaStro07}, \cite[Theorem 4]{KhM} and \cite[Proposition 6.8]{Ma2}, we have the following exact sequences $$
0\rightarrow \mathcal{L}_1Z_sK(x)\rightarrow T_sK(x)\rightarrow T_sG_sK(x)\rightarrow 0,\,\,\,0\rightarrow \mathcal{L}_1Z_s\Delta(x)\rightarrow T_s\Delta(x)\rightarrow T_sG_s\Delta(x)\rightarrow 0.
$$
Combining this with (\ref{eqa11}), we get \begin{equation}\label{eqa12}
[\mL_2Z_sL(x)]=[\mL_1Z_sK(x)]-[\mL_1Z_s\Delta(x)]=[T_sK(x)]-[T_s\Delta(x)] .
\end{equation}

By assumption, $sx>x$, we have $T_sL(x)=0$. Applying Lemma \ref{TsPx}, we see that $\mathcal{L}_1T_sL(x)\cong L(x)\<1\>$. We claim that the canonical map $L(x)\<1\>\cong\mathcal{L}_1T_sL(x)\rightarrow T_sK(x)$ is nonzero.

Suppose that this canonical map is zero. Then we get that the canonical map $T_s K(x)\rightarrow T_s\Delta(x)$ is an isomorphism. That is, $T_s\Delta(x)\cong T_sK(x)$. Using (\ref{eqa12}), we get $\mL_2Z_sL(x)=0$, which is impossible, because (upon forgetting the $\Z$-grading) $\mathcal{L}_2Z_sL(x)$ is isomorphic to $\hat{Z}_sL(x)\cong L(x)$ by \cite{EW} and \cite[Proposition 6.2]{Ma2}. This proves our claim, which means that the canonical map $L(x)\<1\>\cong\mathcal{L}_1T_sL(x)\rightarrow T_sK(x)$ is injective. In this case, we have the following exact sequence \begin{equation}\label{eqa13}
0\rightarrow L(x)\<1\>\cong\mathcal{L}_1T_sL(x)\rightarrow T_sK(x)\rightarrow T_s\Delta(x)\rightarrow 0 .
\end{equation}

Finally, combining (\ref{eqa12}) and (\ref{eqa13}) together we can deduce that $[\mL_2Z_sL(x)]=[L(x)\<1\>]$, hence $\mL_2Z_sL(x)\cong L(x)\<1\>$.
\end{proof}


The following lemma gives a recursive formula to compute the graded character of $\mathcal{L}_1Z_sL(x)$.

\begin{lem}\label{0lem} Let $s\in S$ and $x\in W$. If $sx<x$, then we have  $$\begin{aligned}
{} [\mathcal{L}_1Z_sL(x)]&=v[\Delta(sx)]-v^2[\Delta(x)]-\sum_{\substack{z\in W\\ sz<z>x}}v^{\ell(z)-\ell(x)}P_{x,z}(v^{-2})[\mL_1Z_s L(z)]+(v+1)\sum_{\substack{z\in W\\ sz>z>x}}v^{\ell(z)-\ell(x)}P_{x,z}(v^{-2})[L(z)].
\end{aligned}
$$
\end{lem}

\begin{proof} Let $K(x):=\kor p$, $p:\Delta(x)\rightarrow L(x)$ is the canonical surjection. We have the following exact sequence: \begin{equation}\label{exact2}
\mathcal{L}_1Z_sK(x)\rightarrow\mathcal{L}_1Z_s\Delta(x)\rightarrow \mathcal{L}_1Z_sL(x)\rightarrow Z_sK(x)\overset{Z_s(\iota)}{\rightarrow} Z_s\Delta(x){\rightarrow} Z_sL(x)\rightarrow 0,
\end{equation}
where $\iota: K(x)\rightarrow\Delta(x)$ is the natural embedding.

By assumption, $sx<x$. In this case, we have $Z_s\Delta(x)=0$ because $\Delta(x)$ has a unique simple socle $L(w_0)$ and $sw_0<w_0$. By \cite[Claim 3.2]{MaStro07} we know by that $\mL_1Z_s\Delta(x)\cong\Delta(sx)/\Delta(x)$ upon forgetting the $\Z$-gradings. However, using the short exact sequence \cite[(5.2)]{Stro2} in the graded setting one can check the same argument in the proof of \cite[Claim 3.2]{MaStro07} and \cite[(5.2)]{Stro2} imply that $\mL_1Z_s\Delta(x)\cong\bigl(\Delta(sx)/(\Delta(x)\<1\>)\bigr)\<1\>$. Applying Lemma \ref{keylem3}, we see that $\mL_2 Z_sK(x)=0$ and hence moreover, $$\begin{aligned}
&\quad\,[\mL_1Z_sK(x)]=-([Z_sK(x)]-[\mL_1Z_sK(x)]+[\mL_2Z_sK(x)])+[Z_sK(x)]=-[\mL Z_sK(x)]+[Z_sK(x)]\\
&=-\sum_{\substack{z\in W\\ z>x}}v^{\ell(z)-\ell(x)}P_{x,z}(v^{-2})[\mL Z_sL(z)]+[Z_sK(x)]\\
&=\sum_{\substack{z\in W\\ sz<z>x}}v^{\ell(z)-\ell(x)}P_{x,z}(v^{-2})[\mL_1 Z_sL(z)]-\sum_{\substack{z\in W\\ sz>z>x}}v^{\ell(z)-\ell(x)}P_{x,z}(v^{-2})([L(z)]+v[L(z)])+[Z_sK(x)]\\
&=\sum_{\substack{z\in W\\ sz<z>x}}v^{\ell(z)-\ell(x)}P_{x,z}(v^{-2})[\mL_1Z_s L(z)]-(v+1)\sum_{\substack{z\in W\\ sz>z>x}}v^{\ell(z)-\ell(x)}P_{x,z}(v^{-2})[L(z)]+[Z_sK(x)].
\end{aligned}
$$

Now we consider the following short exact sequence $$
\mL_2Z_sL(x)=0\rightarrow \mathcal{L}_1Z_sK(x)\rightarrow \mathcal{L}_1Z_s\Delta(x)\rightarrow \mathcal{L}_1Z_sL(x)\rightarrow Z_sK(x)\rightarrow 0=Z_s\Delta(x).
$$
We get that  $$
{} [\mathcal{L}_1Z_sL(x)]=v[\Delta(sx)]-v^2[\Delta(x)]-\sum_{\substack{z\in W\\ sz<z>x}}v^{\ell(z)-\ell(x)}P_{x,z}(v^{-2})[\mL_1Z_s L(z)]+(v+1)\sum_{\substack{z\in W\\ sz>z>x}}v^{\ell(z)-\ell(x)}P_{x,z}(v^{-2})[L(z)].
$$
\end{proof}


\noindent
{\textbf{Proof of Theorem \ref{mainthm3}:}} Part (1) of Theorem \ref{mainthm3} has been proved in Lemmas \ref{0lem} and \ref{keylem3}. It remains to show Part (2) of Theorem \ref{mainthm3}.

Assume that in the Grothendieck group of $\O_0^\Z$, $$
[M/\hat{Z}_s(M)]=\sum_{x\in W}c_x(v,v^{-1})[L(x)], $$ where $c_x(v,v^{-1})\in\mathbb{N}[v,v^{-1}]$ for each $x\in W$. We consider the quotient module $M/\hat{Z}_s(M)$. It is clear that $\hat{Z}_s(M/\hat{Z}_s(M))=0$.
By \cite[Proposition 6.7]{Ma2}, we have $\mL_1T_s=\hat{Z}_s$. It follows that $\mL_1T_s(M/\hat{Z}_s(M))=0$. Therefore, $$
[\mL T_s(M/\hat{Z}_s(M))]=[T_s(M/\hat{Z}_s(M))].
$$

On the other hand, by \cite[Corollary 5.8]{AS}, $T_s(\hat{Z}_s(M))=0$. It follows that $T_sM\cong T_s(M/\hat{Z}_s(M))$. Therefore, $$\begin{aligned}
{} [T_sM]&=[T_s(M/\hat{Z}_s(M))]=[\mL T_s(M/\hat{Z}_s(M))]=\sum_{x\in W}c_x(v,v^{-1})[\mL T_sL(x)]\\
&=\sum_{x\in W}c_x(v,v^{-1})[T_sL(x)]-\sum_{x\in W}c_x(v,v^{-1})[\mL_1T_sL(x)]\\
&=\sum_{\substack{x\in W\\ sx<x}}c_x(v,v^{-1})[T_sL(x)]-\sum_{\substack{x\in W\\ sx>x}}vc_x(v,v^{-1})[L(x)] .
\end{aligned}
$$
Hence the theorem follows.  \hfill\qed
\medskip

Let $s\in S$ and $x\in W$. It is well-known that if $sx>x$ then $T_s\Delta(x)\cong\Delta(sx)$. However, if $sx<x$, then the $\Z$-grading structure of $T_s\Delta(x)$ is in general unknown. The following result gives  an answer on the level of Grothendieck groups.

\begin{prop}\label{domVerma} Let $s\in S$ and $x\in W$. Suppose that $sx>x$. Then there is the following exact sequence in $A\gmod$: $$
0\rightarrow\Delta(sx)\<1\>\overset{f}{\rightarrow}\Delta(x)\overset{g}{\rightarrow} T_s\Delta(sx)\overset{h}{\rightarrow} T_s\Delta(x)\<-1\>\rightarrow 0.
$$
In particular, $[T_s\Delta(sx)]=[\Delta(x)]+(v^{-1}-v)[\Delta(sx)]$.
\end{prop}

\begin{proof} If we forget the$\Z$-grading, then the conclusion of the lemma follows from \cite[6.3]{AL}. In other words, we have the following exact sequence of ungraded $A$-module homomorphisms: $$
0\rightarrow\Delta(sx)\overset{f'}{\rightarrow}\Delta(x)\overset{g'}{\rightarrow} T_s\Delta(sx)\overset{h'}{\rightarrow} T_s\Delta(x)\rightarrow 0.
$$
Note that $\dim\Hom_A(\Delta(sx),\Delta(x))=1$ and there is an injective degree $0$ homomorphism $f:\Delta(sx)\<1\>\hookrightarrow\Delta(x)$. It follows that $f'$ has to be a scalar multiple of $f$ and in particular homogeneous. Similarly, as $$
\dim\Hom_A\bigl(T_s\Delta(sx),T_s\Delta(x)\bigr)=\dim\Hom_A\bigl(\Delta(sx),\Delta(x)\bigr)=1 ,
$$
we can deduce that $h'$ is homogeneous of degree $1$ as well. We claim that $\dim\Hom_A(\Delta(x),T_s\Delta(sx))=1$.

Forgetting the $\Z$-grading, we can deduce that \begin{equation}\label{sDeltasx}
\,[T_s\Delta(sx)]|_{v=1}=[(\mathcal{L}T_s)\Delta(sx)]|_{v=1}=[(\mathcal{L}T_s)\nabla(sx)]|_{v=1}=[T_s\nabla(sx)]|_{v=1}-[(\mathcal{L}_1T_s)\nabla(sx)]|_{v=1}=[\nabla(x)]|_{v=1},
\end{equation}
which implies that $\dim\Hom_A(P(x),T_s\Delta(sx))=1$ and hence the ungraded composition multiplicity of $L(x)$ in $T_s\Delta(sx)$ is one. As $\Hom_A(\Delta(x),T_s\Delta(sx))\hookrightarrow\Hom_A(P(x),T_s\Delta(sx))$, we can now deduce that
$$\Hom_A(\Delta(x),T_s\Delta(sx))=1, $$ which implies that there exists a nonzero homogeneous homomorphism from $\Delta(x)$ to $T_s\Delta(sx)$.
Applying Theorem \ref{mainthm2}, we see that $L(x)$ appears as a unique graded composition factor in $T_s\Delta(sx)$ because $T_s\Delta(sx)$ maps onto $T_sL(sx)$ and (\ref{sDeltasx}). Hence the degree of this nonzero homogeneous homomorphism is zero. This proves our claim and hence we complete the proof of the proposition.
\end{proof}
\bigskip

\section{A categorical action of Hecke algebra on derived category via derived shuffling functors}

The purpose of this section is to show that there is a categorical action of the Hecke algebra $\HH(W)$ on the derived category $D^b(\O_0^{\Z})$ via derived shuffling functors.

The shuffling functor $C_s$ corresponding to a simple reflection $s$ is the endofunctor of $\O_0$ defined as the cokernel of the adjunction
morphism from the identity functor to the projective functor $\theta_s$, see \cite{Ca} and \cite{MaStro07}. Following \cite[\S2.7]{CMZ}, the graded
lift of $C_s$ is defined by the exact sequence \begin{equation}\label{adj}
\id\<1\>\overset{\adj_s}{\rightarrow}\theta_s\rightarrow C_s\rightarrow 0 .
\end{equation}

For any $w\in W$ with reduced expression $w = s_1s_2\cdots s_m$, we define the functor \begin{equation}\label{braid12}
C_w:=C_{s_m}\cdots C_{s_2}C_{s_1} .
\end{equation}
By \cite{MaStro05}, \cite{MOS}, \cite{KhM}, the resulting functor $C_w$ does not depend on the choice of the reduced expression $s_1s_2\cdots s_m$ of $w$. The functor $C_w$ is right exact and corresponding left derived functor $\mL C_w$ is an auto-equivalence of $D^b(\O_0^{\Z})$.

\begin{lem}\label{Csaction0} Let $x\in W$ and $s\in S$. There are the following isomorphisms in $\O_0^\Z$: $$
C_s\nabla(x)\cong\begin{cases}\nabla(x)\<-1\>, &\text{if $x<xs$;}\\ \nabla(xs), &\text{if $x>xs$.}\end{cases}
$$
Moreover, if $xs>x$, then $C_s\Delta(x)\cong\Delta(xs)$.
\end{lem}

\begin{proof} If $x>xs$, then by  \cite[Theorem 3.10]{Stro2} we have that $C_s\nabla(x)\cong \nabla(xs)$.

Now assume $x<xs$. Note that the adjunction map $\nabla(x)\<1\>\rightarrow\theta_s\nabla(x)$ can factor through $\nabla(x)\<1\>\twoheadrightarrow\nabla(xs)$ as $$
\nabla(x)\<1\>\twoheadrightarrow\nabla(xs)\overset{k'}{\hookrightarrow}\theta_s\nabla(x),
$$
where $k': \nabla(xs){\hookrightarrow}\theta_s\nabla(x)$ is the same map given in \cite[(5.3)]{Stro2}. Therefore, it follows from \cite{Stro2} that $C_s\nabla(x)\cong\nabla(x)\<-1\>$ in this case. By \cite[(5.2)]{Stro2}, we have a short exact sequence $$
0\rightarrow\Delta(x)\rightarrow\theta_s\Delta(xs)\rightarrow\Delta(xs)\<-1\>\rightarrow 0 .
$$
Applying \cite[Corollary 5.5]{Stro2}, we get that $$
\theta_s\Delta(xs)\cong\theta_s\Delta(x)\<-1\> .
$$
It follows that $C_s\Delta(x)\cong\Delta(xs)$ in this case. This completes the proof of the lemma.
\end{proof}

\begin{lem}\text{\rm (\cite{MaStro05})}\label{Csaction1} Let $s\in S$.

(1) For any $x\in W$ and $i>0$ we have $\mathcal{L}_iC_s\Delta(x)=0$;

(2) For any $i>1$ we have $\mathcal{L}_iC_s=0$.

(3) For any $x\in W$ we have $$
\mathcal{L}_1C_s\nabla(x)=\ker\bigl(\adj_s\!\nabla(x)\bigr).
$$
where $\adj_s$ is defined as in (\ref{adj}).
\end{lem}

\begin{proof} Parts (1) and (2) follow from \cite[Proposition 5.3]{MaStro05}. Part (3) follows from the same argument used in the proof of \cite[Proposition 5.3(3)]{MaStro05}.
\end{proof}

\begin{lem}\label{Csaction2} Let $s\in S$ and $x\in W$. Then

(1) if $xs<x$ then $$ [\theta_s\nabla(x)]=v[\nabla(x)]+[\nabla(xs)],\,\,[C_s\nabla(x)]=[\nabla(xs)].
$$
and $\mathcal{L}_1C_s\nabla(x)=0$;

(2) if $xs>x$ then  $$ [\theta_s\nabla(x)]=v^{-1}[\nabla(x)]+[\nabla(xs)],\,\,[C_s\nabla(x)]=v^{-1}[\nabla(x)].
$$ and $$
[\mathcal{L}_1C_s\nabla(x)]=v[\nabla(x)]-[\nabla(xs)]=[K_{x,xs}\<1\>],
$$
where $K_{x,xs}$ denotes the kernel of the (unique up to a scalar) nontrivial surjective homomorphism $\nabla(x)\twoheadrightarrow\nabla(xs)\<-1\>$ in the case $x<xs$.
\end{lem}

\begin{proof} By the proof of Lemma \ref{Csaction1}, $\mathcal{L}_1C_s\nabla(x)=\ker\bigl(\adj_s\!\nabla(x)\bigr)$ is equal to the kernel of the canonical map  $\nabla(x)\twoheadrightarrow\nabla(xs)\<-1\>$, which implies that $$
[\mathcal{L}_1C_s\nabla(x)]=v[\nabla(x)]-[\nabla(xs)]=[K_{x,xs}\<1\>]. $$
The remaining statements follow from \cite[Theorem 3.10, (5.3)]{Stro2}.
\end{proof}

\noindent
{\textbf{Proof of Theorem \ref{mainthm4}:}}  We first show that for any $x\in W$, \begin{equation}\label{eq41}
(\mL C_s-v^{-1})(\mL C_s+v)[\nabla(x)]=0 .
\end{equation}

Suppose $xs>x$. Then by Lemmas \ref{Csaction0}, \ref{Csaction1} and \ref{Csaction2}, we have $$
(\mL C_s+v)[\nabla(x)]=[\nabla(x)\<-1\>]-[K_{x,xs}\<1\>]+v[\nabla(x)]=v^{-1}[\nabla(x)]+[\nabla(xs)].
$$
Thus, $$\begin{aligned}
&\quad (\mL C_s-v^{-1})(\mL C_s+v)[\nabla(x)]\\
&=(\mL C_s-v^{-1})\bigl(v^{-1}[\nabla(x)]+[\nabla(xs)]\bigr)\\
&=v^{-1}[\mL C_s\nabla(x)]-v^{-2}[\nabla(x)]+[\mL C_s\nabla(xs)]-v^{-1}[\nabla(xs)]\\
&=v^{-2}[\nabla(x)]-([\nabla(x)]-v^{-1}[\nabla(xs)])-v^{-2}[\nabla(x)]+[\nabla(x)]-v^{-1}[\nabla(xs)]\\
&=0.
\end{aligned}
$$
Now suppose that $xs<x$. Then by Lemmas \ref{Csaction0}, \ref{Csaction1} and \ref{Csaction2}, we have that $$
(\mL C_s+v)[\nabla(x)]=[\nabla(xs)]+v[\nabla(x)].
$$
Thus, $$\begin{aligned}
&\quad (\mL C_s-v^{-1})(\mL C_s+v)[\nabla(x)]\\
&=(\mL C_s-v^{-1})\bigl([\nabla(xs)]+v[\nabla(x)]\bigr)\\
&=[\mL C_s\nabla(xs)]-v^{-1}[\nabla(xs)]+v[\mL C_s\nabla(x)]-[\nabla(x)]\\
&=v^{-1}[\nabla(xs)]-(v[\nabla(xs)]-[\nabla(x)])-v^{-1}[\nabla(xs)]+v[\nabla(xs)]-[\nabla(x)]\\
&=0.
\end{aligned}
$$
This completes the proof of (\ref{eq41}).

Second, we want to show that for any $u,w\in W$ with $\ell(uw)=\ell(u)+\ell(w)$, $\mL C_u\mL C_w=\mL C_{uw}$ on the Grothendieck group of $D^b(\O_0^\Z)$. Using Lemma \ref{Csaction1}, it suffices to show that for any $x\in W$,
\begin{equation}\label{eq22}
[C_u C_w\Delta(x)]=[C_{uw}(\Delta(x))].
\end{equation}
However, this follows from (\ref{braid12}). Now to complete the proof of the first part of the theorem, it remains to show that
$[(\mL C_s)\nabla(x)]=H_{w_0x}H_s,\forall\,s\in S, x\in W$.

Let $s\in S$ and $x\in W$. Suppose $xs<x$. Then $w_0xs>w_0x$. Applying Lemma \ref{Csaction1}, we get that $C_s\nabla(x)\cong\nabla(xs)$. On the other hand, we have $$
H_{w_0x}H_s=H_{w_0xs}.
$$
Hence $[(\mL C_s)\nabla(x)]=[\nabla(xs)]=H_{w_0xs}=H_{w_0x}H_s$.

Now suppose that $xs>x$. Then $w_0xs<w_0x$. In this case, applying Lemma \ref{Csaction2} we can deduce that $$\begin{aligned}
\bigl[{\mL}C_s\nabla(x)\bigr] &=[C_s\nabla(x)]-[\mL_1 C_s\nabla(x)]=[\nabla(x)\<-1\>]-[K_{x,xs}\<1\>]\\
&=v^{-1}[\nabla(x)]-\bigl(v[\nabla(x)]-[\nabla(xs)]\bigr)\\
&=[\nabla(xs)]+(v^{-1}-v)[\nabla(x)] .
\end{aligned}
$$
On the other hand, the fact that $w_0sx<w_0x$ implies that $$
H_{w_0x}H_s=(v^{-1}-v)H_{w_0x}+H_{w_0xs}.
$$
This proves that $[(\mL C_s)\nabla(x)]=H_{w_0x}H_s$

By (\ref{Lnablax}), we have \begin{equation}\label{Lnablax2}
[L(x)]=[\nabla(x)]+\sum_{y>x}(-v)^{\ell(x)-\ell(y)}P_{w_0y,w_0x}(v^{2})[\nabla(y)].
\end{equation}
Hence $$\begin{aligned}
\rho([L(x)])&=H_{w_0x}+\sum_{y>x}(-v)^{\ell(x)-\ell(y)}P_{w_0y,w_0x}(v^{2})H_{w_0y}\\
&=H_{w_0x}+\sum_{w_0y<w_0x}(-v)^{\ell(w_0y)-\ell(w_0x)}P_{w_0y,w_0x}(v^{2})H_{w_0y}\\
&=\ucH_{w_0x} .
\end{aligned}$$
This completes the proof of Theorem \ref{mainthm4}.\hfill\qed
\bigskip

\noindent
{\textbf{Proof of Proposition \ref{mainprop6}:}} Assume $C_sL(x)\neq 0$. Then by definition of $C_s$ we can deduce that $\theta_sL(x)\neq 0$. Hence $\huH_x\uH_s\neq 0$ by Lemma \ref{keylem0}. Applying \cite[(5.1.14)]{L0}, we get that $s\geq_{L}x$, where $\leq_L$ is the Kazhdan-Lusztig left preorder defined in \cite{KL}. It follows that
$s\in \mathscr{R}(x)$. That is, $xs<x$. Conversely, assume $xs<x$. Then by \cite[Proposition 46]{CMZ} and the definition of $C_s$ we can deduce that $C_sL(x)\neq 0$.

Now assume $xs<x$. Applying Lemma \ref{Csaction2}, we can deduce that $\mathcal{L}_1C_s\nabla(x)=0$. Note that
$\mathcal{L}_2C_s(\nabla(x)/L(x))=0$. It follows that $\mathcal{L}_1C_sL(x)=0$.

The assumption that $xs<x$ implies that $w_0xs>w_0x$. Applying Theorem \ref{mainthm4} and (\ref{2Hs}), we get that $$\begin{aligned}
{}\rho([C_sL(x)])&=\ucH_{w_0x}H_s=\ucH_{w_0x}\ucH_s+v^{-1}\ucH_{w_0x}\\
&=v^{-1}\ucH_{w_0x}+\ucH_{w_0xs}+\sum_{\substack{y\in W\\ w_0ys<w_0y<w_0x}}\mu(w_0y,w_0x)\ucH_{w_0y}\\
&=v^{-1}\ucH_{w_0x}+\ucH_{w_0xs}+\sum_{\substack{y\in W\\ ys>y>x}}\mu(x,y)\ucH_{w_0y}\\
&=v^{-1}[L(x)]+[L(xs)]+\sum_{\substack{y\in W\\ ys>y>x}}\mu(x,y)[L(y)].
\end{aligned} $$

On the other hand, applying Theorem \ref{mainthm1}, we get that $$\begin{aligned}
{}\rho([C_sL(x)])&=[L(x)]H_s\\
&=\sum_{y\geq x}(-v)^{\ell(x)-\ell(y)}P_{w_0y,w_0x}(v^2)[\nabla(y)]H_s\\
&=\sum_{\substack{y\geq x\\ ys<y}}(-v)^{\ell(x)-\ell(y)}P_{w_0y,w_0x}(v^2)H_{w_0y}H_s+\sum_{\substack{y\geq x\\ ys>y}}(-v)^{\ell(x)-\ell(y)}P_{w_0y,w_0x}(v^2)H_{w_0y}H_s\\
&=\sum_{\substack{y\geq x\\ x\nleq ys<y}}(-v)^{\ell(x)-\ell(y)}P_{w_0y,w_0x}(v^2)H_{w_0ys}+\sum_{\substack{y\geq x\\ x\leq ys<y}}(-v)^{\ell(x)-\ell(y)}P_{w_0y,w_0x}(v^2)H_{w_0ys}\\
&\qquad +\sum_{\substack{y\geq x\\ ys>y}}(-v)^{\ell(x)-\ell(y)}P_{w_0y,w_0x}(v^2)H_{w_0y}H_s\\
&=\sum_{\substack{y\geq x\\ x\nleq ys<y}}(-v)^{\ell(x)-\ell(y)}P_{w_0y,w_0x}(v^2)H_{w_0ys}+\sum_{\substack{y\geq x\\  ys>y}}(-v)^{\ell(x)-\ell(y)-1}P_{w_0ys,w_0x}(v^2)H_{w_0y}\\
&\qquad +\sum_{\substack{y\geq x\\ ys>y}}(-v)^{\ell(x)-\ell(y)}P_{w_0y,w_0x}(v^2)\bigl((v^{-1}-v)H_{w_0y}+H_{w_0ys}\bigr)\\
\end{aligned}
$$
Applying \cite[(2.3.g)]{KL} and \cite[Corollary 4.4]{Bre}, we see that for any $y,w\in W$ with $y<w, ys<y, ws>w$, $$
P_{y,w}(v^2)=P_{ys,w}(v^2) .
$$
Therefore, $$\begin{aligned}
{}\rho([C_sL(x)])&=\sum_{\substack{y\geq x\\ x\nleq ys<y}}(-v)^{\ell(x)-\ell(y)}P_{w_0y,w_0x}(v^2)H_{w_0ys}+\sum_{\substack{y\geq x\\ ys>y}}(-v)^{\ell(x)-\ell(y)}P_{w_0y,w_0x}(v^2)\bigl(-vH_{w_0y}+H_{w_0ys}\bigr)\\
&=\sum_{\substack{y\geq x\\ x\nleq ys<y}}(-v)^{\ell(x)-\ell(y)}P_{w_0y,w_0x}(v^2)H_{w_0ys}+\sum_{\substack{y\geq x\\ ys>y}}(-v)^{\ell(x)-\ell(y)+1}\biggl(P_{w_0y,w_0x}(v^2)(H_{w_0y}-v^{-1}H_{w_0ys})\biggr).
\end{aligned}
$$
It follows that
$$
[C_sL(x)]=\sum_{\substack{y\geq x\\ x\nleq ys<y}}(-v)^{\ell(x)-\ell(y)}P_{w_0y,w_0x}(v^2)[\nabla(ys)]+\sum_{\substack{y\geq x\\ ys>y}}(-v)^{\ell(x)-\ell(y)+1}P_{w_0y,w_0x}(v^2)\bigl([\nabla(y)]-v^{-1}[\nabla(ys)]\bigr) .
$$
This completes the proof of Proposition \ref{mainprop6}.
\hfill\qed

\end{document}